\documentclass[11pt]{article}
\usepackage{latexsym,amsfonts,amssymb,amsmath,amsthm}
\usepackage{graphicx}
\usepackage{cite}
\usepackage[usenames,dvipsnames]{color}
\usepackage{ulem}

\begin{document}
\setlength{\baselineskip}{16pt}

\parindent 0.5cm
\evensidemargin 0cm \oddsidemargin 0cm \topmargin 0cm \textheight
22cm \textwidth 16cm \footskip 2cm \headsep 0cm

\newtheorem{thm}{Theorem}[section]
\newtheorem{cor}{Corollary}[section]
\newtheorem{lem}{Lemma}[section]
\newtheorem{prop}{Proposition}[section]
\newtheorem{defn}{Definition}[section]
\newtheorem{rk}{Remark}[section]
\newtheorem{nota}{Notation}[section]
\newtheorem{Ex}{Example}[section]
\def\nm{\noalign{\medskip}}

\numberwithin{equation}{section}

\def\p{\partial}
\def\I{\textit}
\def\R{\mathbb R}
\def\C{\mathbb C}
\def\u{\underline}
\def\l{\lambda}
\def\a{\alpha}
\def\O{\Omega}
\def\e{\epsilon}
\def\ls{\lambda^*}
\def\D{\displaystyle}
\def\wyx{ \frac{w(y,t)}{w(x,t)}}
\def\imp{\Rightarrow}
\def\tE{\tilde E}
\def\tX{\tilde X}
\def\tH{\tilde H}
\def\tu{\tilde u}
\def\d{\mathcal D}
\def\aa{\mathcal A}
\def\DH{\mathcal D(\tH)}
\def\bE{\bar E}
\def\bH{\bar H}
\def\M{\mathcal M}
\renewcommand{\labelenumi}{(\arabic{enumi})}

\def\disp{\displaystyle}
\def\undertex#1{$\underline{\hbox{#1}}$}
\def\card{\mathop{\hbox{card}}}
\def\sgn{\mathop{\hbox{sgn}}}
\def\exp{\mathop{\hbox{exp}}}
\def\OFP{(\Omega,{\cal F},\PP)}
\newcommand\JM{Mierczy\'nski}
\newcommand\RR{\ensuremath{\mathbb{R}}}
\newcommand\CC{\ensuremath{\mathbb{C}}}
\newcommand\QQ{\ensuremath{\mathbb{Q}}}
\newcommand\ZZ{\ensuremath{\mathbb{Z}}}
\newcommand\NN{\ensuremath{\mathbb{N}}}
\newcommand\PP{\ensuremath{\mathbb{P}}}
\newcommand\abs[1]{\ensuremath{\lvert#1\rvert}}

\newcommand\normf[1]{\ensuremath{\lVert#1\rVert_{f}}}
\newcommand\normfRb[1]{\ensuremath{\lVert#1\rVert_{f,R_b}}}
\newcommand\normfRbone[1]{\ensuremath{\lVert#1\rVert_{f, R_{b_1}}}}
\newcommand\normfRbtwo[1]{\ensuremath{\lVert#1\rVert_{f,R_{b_2}}}}
\newcommand\normtwo[1]{\ensuremath{\lVert#1\rVert_{2}}}
\newcommand\norminfty[1]{\ensuremath{\lVert#1\rVert_{\infty}}}

\title{Non-wandering points for autonomous/periodic parabolic equations on the circle}

\author {
\\
Wenxian Shen\thanks{Partially supported by the NSF grant DMS--1645673.}\\
Department of Mathematics and Statistics\\
 Auburn University, Auburn, AL 36849, USA
\\
\\
Yi Wang\thanks{Partially supported by NSF of China No.11825106, 11771414, Wu Wen-Tsun Key Laboratory,  University of Science and Technology of China.} \\
School of Mathematical Science\\
 University of Science and Technology of China
\\ Hefei, Anhui, 230026, P. R. China
\\
\\
 Dun Zhou\thanks{Partially supported by NSF of China No.11971232, 11601498 and the Fundamental Research Funds for the Central Universities No.30918011339.}\\
 School of Science
 \\ Nanjing University of Science and Technology
 \\ Nanjing, Jiangsu, 210094, P. R. China
 \\
\\
}
\date{}

\maketitle
% insert the table of contents
%\tableofcontents

%---------------------SECTION DIVIDE LINE---------------------------
\begin{abstract}
We study the properties of non-wandering points of the following scalar reaction-diffusion equation on the circle $S^1$,
\begin{equation*}
u_{t}=u_{xx}+f(t,u,u_{x}),\,\,t>0,\,x\in S^{1}=\mathbb{R}/2\pi \mathbb{Z},
\end{equation*}
where $f$ is independent of $t$ or $T$-periodic in $t$.  Assume that the equation admits a compact global attractor. It is proved that, any non-wandering point is a limit point of the system (that is, it is a point in some $\omega$-limit set). More precisely, in the autonomous case, it is proved that
any non-wandering point is either a fixed point or generates a rotating wave on the circle.
In the periodic case, it is proved that any non-wandering point is a periodic point or generates a rotating wave on a torus.
 In  particular, if $f(t,u,-u_x)=f(t,u,u_x)$, then any  non-wandering point is a fixed point in the autonomous case, and is a periodic point in the periodic case.
\end{abstract}

\section{Introduction}

Non-wandering set, as a common invariant set, plays an important role in the characterization of the overall behavior of a dynamical system. For a dynamical system on a compact space, the non-wandering set, which is a non-empty closed invariant set, is the hub of recurrence behavior, since it contains all $\omega$-limit and $\alpha$-limit points and recurrent points, including, naturally all periodic orbits. All orbits will stay in any neighborhood of the non-wandering set as the time tends to infinity. Conversely, a non-wandering point may not be a limit point, see an example in the Appendix.
From the view of topological dynamical systems, for a compact dynamical system, by using Birkhoff Ergodic Theorem, the supports of invariant probability measures are contained in the non-wandering set; and hence by the Variational Principle, the topological entropy of the system equals to the topological entropy restricted to the non-wandering set (see \cite{Katok-Has,Ruelle89}). From the view of differentiable dynamical systems, the structure of a non-wandering set is strongly related to the structural stability and $\Omega$-stability of the system (see \cite{Palis1968,Palis-Smale,Smale}).
 There are many works concerning with the structure of non-wandering sets and its related properties (see, e.g., \cite{ABD,BCMN,CN,Will,Ye} and the references therein).

Although, Andronov et al. \cite{ALGM} have already given a classification for the non-wandering points of the two dimensional autonomous system. Unfortunately, there is no general way to characterize the non-wandering points for higher dimensional systems. Therefore, it is necessary and important to select some typical systems to study the structure of its non-wandering set.

A typical model of infinite dimensional dynamical system is the following reaction diffusion equation on the circle $S^1$,
\begin{equation}\label{general-equa}
u_{t}=u_{xx}+f(t,u,u_{x}),\,\,t>0,\,x\in S^{1}=\mathbb{R}/2\pi \mathbb{Z}.
\end{equation}

Our goal is to thoroughly understand the structure of the non-wandering set of the above equation \eqref{general-equa}, in the following two circumstances:
\begin{equation}\label{autonomous}
  f(t,u,u_x)=f(u,u_x),\quad \forall\,  t\in \mathbb{R},
\end{equation}
or
\begin{equation}\label{periodic-eq}
 f(t+T,u,u_x)=f(t,u,u_x),\quad \forall\,  t\in \mathbb{R},\ \text{and some }T>0.
\end{equation}

In the autonomous case, that is, $f(t,u,u_x)$ is as in \eqref{autonomous},
the study of the asymptotic behavior of bounded solutions of \eqref{general-equa} can be originated from the introduction of zero number (or lap number) function in Angenent and Fiedler \cite{Angenent1988}, Massatt \cite{Massatt1986} and Matano \cite{Matano}. In \cite{Angenent1988}, \cite{Massatt1986} and \cite{Matano}, it was  proved independently that any $C^1$-bounded solution of \eqref{general-equa} approaches a set of functions of the form
\[
\Sigma \phi=\{\phi(\cdot+a)\,|\, a\in S^1\},
\]
where $\phi$ satisfies the following equation
\begin{equation*}
  \phi_{xx}+c\phi_x+f(\phi,\phi_x)=0,\quad x\in S^1
\end{equation*}
for some $c\in \mathbb{R}$. In other words, any $C^1$-bounded solution of \eqref{general-equa} approximates either a constant equilibrium (i.e., $\phi$ is constant and $c=0$), a rotating wave (i.e., $\phi$ is nonconstant and $c\neq 0$) or a one-dimensional manifold of standing waves (i.e., $\phi$ is nonconstant and $c=0$). Particularly, if $f(u,u_x)=f(u,-u_x)$, then any bounded solution is closing to an equilibrium (see \cite{Matano}). Later, Fiedler and Mallet-Paret \cite{Fiedler} generalized these results to $f$ in \eqref{autonomous} depending on $x$ ($f$ is spatially-inhomogeneous), and they obtained the celebrated Poincar\'{e}-Bendixson type Theorem: any $\omega$-limit set of  \eqref{general-equa} is either a single periodic orbit or it consists of equilibria and connecting (homoclinic and heteroclinic) orbits.
 It should also be pointed out that, in the situation that $f$ depends on $x$,
 transversality of the stable and unstable manifolds of hyperbolic equilibria and periodic orbits, and generic hyperbolicity have been established in \cite{CR,JR1}. After then, Joly and Raugel \cite{JR2} proved the generic Morse-Smale property for the spatially-inhomogeneous system with a compact global attractor.

Assume that \eqref{general-equa}+\eqref{autonomous} admits a compact global attractor $\mathcal{A}$,  Matano and Nakamura \cite{Ma-Na},  and Fiedler, Rocha and Wolfrum \cite{FRW} investigated the structure of $\mathcal{A}$ in succession. In \cite{Ma-Na}, Matano and Nakamura obtained that $\mathcal{A}$ is the graph of a $2[N/2]+1$-dimensional mapping, where $N$ is the maximal value of the generalized Morse index of equilibria and periodic solutions, and they also proved that there exists no homoclinic orbit or heteroclinic cycle.  Fiedler, Rocha and Wolfrum in \cite{FRW} established a necessary and sufficient condition for the existence of a heteroclinic connecting orbits between hyperbolic equilibria or rotating waves.

 By the work of Joly and Raugel \cite{JR2}, for a {\it generic} $f=f(x,u,u_x)$ (in the sense of $C^2$-Whitney topology) such that  \eqref{general-equa}+\eqref{autonomous} admits a compact global attractor,  the  non-wandering set of  \eqref{general-equa}+\eqref{autonomous} consists of finitely many  hyperbolic equilibria
 and \eqref{general-equa}+\eqref{autonomous} is of the so called Morse-Smale property. It  needs to emphasize that, from the view of the equation, for a given $f$, it is unclear whether {it} belongs to that ``generic'' point. Thus, an ideal state is that we know the structure of the non-wandering set of \eqref{general-equa} for general $f$. However, to the best of our knowledge, there is still no such specific description at all, up to now.

In the time periodic case, that is $f(t,u,u_x)$ satisfying
\eqref{periodic-eq}, Sandstede and Fiedler \cite{SF1} proved  that any  $\omega$-limit set of \eqref{general-equa}
can be viewed as a subset of the two-dimensional torus $\mathcal{T}^1\times S^1$ carrying a linear flow. Particularly, if $f(t,u,u_x)=f(t,u,-u_x)$, then any $\omega$-limit set of   \eqref{general-equa} is  a periodic orbit (see Chen and Matano \cite{Chen1989160}).
{Yet}, it remains unknown for the structure of non-wandering set for \eqref{general-equa}+\eqref{periodic-eq}
 and whether \eqref{general-equa}+\eqref{periodic-eq}  is of the so called Morse-Smale property.
 It should be pointed out that  the dynamics of \eqref{general-equa} with  the general time periodic nonlinearity $f=f(t,x,u,u_x)$ could be very complicated. Actually, Sandstede and Fiedler \cite{SF1} have shown that any time-periodic planar vector field can  be embedded into \eqref{general-equa}
 with certain nonlinearity $f=f(t,x,u,u_x)$ (see also the comments in \cite{Hale}).

 The current paper is devoted to investigating the structure of the non-wandering set for \eqref{general-equa} in the autonomous case \eqref{autonomous}, as well as the periodic case \eqref{periodic-eq}. To carry out our study, we assume throughout the rest of the paper that $f\in C^2(\mathbb{R}\times\mathbb{R}\times\mathbb{R},\mathbb{R})$ and $X$ is the fractional power space associated with the operator $u\rightarrow -u_{xx}:H^{2}(S^{1})\rightarrow L^{2}(S^{1})$, denoted by $A_0$, satisfies $X\hookrightarrow C^{1}(S^{1})$ (that is, $X$ is compactly embedded in $C^{1}(S^{1})$). By the standard theory for parabolic equations (see \cite{Hen}), for any $u_0\in X$, \eqref{general-equa} defines (locally) a unique solution $u(t,\cdot;u_0)$ in $X$ with $u(0,x;u_0)=u_0(x)$. It follows from \cite{Hen} (see also \cite{Hess,Mierczynski}) and the standard a priori estimates for parabolic equations, if $u(t,\cdot;u_0)$ is bounded in $X$ in the existence interval of the solution, then $u$ is a globally defined classical solution.

 Among others, we will prove

\begin{itemize}
 \item[$\bullet$] (see {\bf Theorem} \ref{aut-thm}) {\it Assume that {\rm \eqref{general-equa}+\eqref{autonomous}} admits a compact global attractor
 and $u_0$ is a non-wandering point of {\rm \eqref{general-equa}+\eqref{autonomous}}. Then $u_0$ is an $\omega$-limit point, and hence is either  a fixed point (i.e., $u(t,x;u_0)\equiv u_0(x)$)
 or generates a rotating wave
 on the circle (i.e., $u_0(x)$ is not a constant function and $u(t,x;u_0)=u_0(x-ct)$ for some  $c\not =0$).
 Particularly, if $f(u,-u_x)=f(u,u_x)$, then $u_0$ is a fixed point.}

 \item[$\bullet$] (see {\bf Theorem} \ref{peri-thm}) {\it Assume that  the associated Poincar\'{e} map $P$ of {\rm \eqref{general-equa}+\eqref{periodic-eq}} admits a compact global attractor and $u_0$ is a non-wandering point of $P$, then
      $u_0$ is an $\omega$-limit point of $P$, and hence is either a periodic point of {\rm \eqref{general-equa}+\eqref{periodic-eq}} (i.e.  $Pu_0=u_0$) or generates a rotating wave on the torus (i.e. $u_0(x)$ is not a constant function and  $(Pu_0)(x)=u_0(x-r)$ for some $r\not =0$).
      Particularly, if $f(t,u,-u_x)=f(t,u,u_x)$, then $u_0$ is a fixed point of $P$.}
\end{itemize}

We remark that the authors of \cite{JR2} obtained some  characterization of the structure of non-wandering set
of  \eqref{general-equa}+\eqref{autonomous} as well as the Morse-Smale property of  \eqref{general-equa}+\eqref{autonomous}  when $f=f(x,u,u_x)$ is such that  \eqref{general-equa}+\eqref{autonomous} admits a compact global attractor;  any equilibrium point or periodic point is hyperbolic;  there is no homoclinic orbit; and  all the heteroclinic orbits are transversal. Therefore, the characterization of non-wandering set of \eqref{general-equa}+\eqref{autonomous} established in this paper is more general than that in \cite{JR2}, {when $f=f(u,u_x)$}. We also remark that  the results established in this paper for the periodic system \eqref{general-equa}+\eqref{periodic-eq}
 pave the way to prove the Morse-Smale property of this system. The reader is referred to Section 5 for more concluding remarks on our results.

The rest of this paper is organized as follows. In Section 2, we summarize some necessary concepts and properties to be used in the proofs of the main results. In Sections 3 and  4, we give characterizations for non-wandering points of \eqref{general-equa} in autonomous and periodic cases, respectively. We provide in Section 5 several concluding remarks on the assumptions,  the main results, and the techniques established in this paper.
In the Appendix, we present an example which shows for a general dynamical system, the non-wandering set and the union of all the limit sets may not be the same.

\section{Preliminaries}
In this section, we introduce some necessary concepts and properties to be used in the later.
We first list some properties of the zero number function for some associated linear equation of \eqref{general-equa}; and then generalize the roughness of exponential dichotomy on the same base flow to different base flows; finally we introduce the so called Sacker-Sell spectrum, invariant spaces and local invariant manifolds associated with \eqref{general-equa}, and the already established relationship between zero number function and these invariant spaces.

\subsection{Zero number function for linear parabolic equations on $S^1$}
Given a $C^{1}$-smooth function $u:S^{1}\rightarrow \mathbb{R}$, the zero number of $u$ is defined as
$$Z(u(\cdot))={\rm card}\{x\in S^{1}:u(x)=0\}.$$
Consider the following linear equation:
\begin{equation}\label{linear-equa}
v_t= v_{xx}+a(t,x)v_x+b(t,x)v,\quad x\in S^1,
\end{equation}
where $a,b $ are bounded continuous functions.

The following lemma was originally appeared in \cite{2038390} and \cite{H.MATANO:1982}, and improved in \cite{Chen98}.

\begin{lem}\label{zero-number}
Let $v(t,x)$ be a nontrivial solution of \eqref{linear-equa} with $v(0,x)=v_0(x)\in H^1(S^1)$.
Then the following properties holds:
\begin{itemize}
\item[{\rm (a)}] $Z(v(t,\cdot))<\infty$ for any $t>0$ and is non-increasing in $t$;

\item[{\rm (b)}]  $Z(v(t,\cdot))$ drops at $t_{0}$ if, and only if, $v(t_{0},\cdot)$ has a
multiple zero in $S^{1}$;

\item[{\rm (c)}]  $Z(v(t,\cdot))$ can drop only finite many times, and there exists a $T_0>0$
such that $v(t,\cdot)$ has only simple zeros in $S^{1}$ as $t\geq T_0$ (hence
$Z(v(t,\cdot))=constant$ as $t\geq T_0$).
\end{itemize}
\end{lem}

\subsection{Exponential dichotomy for linear parabolic equations on $S^1$}
Hereafter, we always assume that $X$ is a fractional power space associated with the operator $u\rightarrow
-u_{xx}:H^{2}(S^{1})\rightarrow L^{2}(S^{1})$, denoted by $A_0$, satisfying $X\hookrightarrow C^{1}(S^{1})$ (that is, $X=D(A_0^\alpha)$ for some $0<\alpha<1$ and is compactly embedded into $C^{1}(S^{1})$).

{Assume moreover in equation \eqref{linear-equa}, $b(t,x)$ is uniformly continuous on $\mathbb{R}\times S^1$, $a(t,x)\in C^1(\mathbb{R}\times S^1)$, $a_t(t,x)$ and $a_x(t,x)$ are bounded.} Define $H(a,b)={\rm cl}\{(a\cdot \tau,b\cdot \tau)=(a(t+\tau,\cdot),b(t+\tau,\cdot)):\tau\in \RR\}$, where the closure is taken in the compact open topology. By the Ascoli--Arzela theorem, $H(a,b)$ is a compact metric space. Let $\Omega=H(a,b)$, the action of time-translation defines a compact flow on $\Omega$. Denotes the element in $\Omega$ by $\omega=(\omega_a,\omega_b)$. Then, equation \eqref{linear-equa} gives rise to a family of equations,
\begin{equation}\label{linear-equa-induced}
v_t=v_{xx}+a(\omega\cdot t,x)v_x+b(\omega\cdot t,x)v,\,\,t>0,\,x\in S^{1}=\mathbb{R}/2\pi \mathbb{Z},
\end{equation}
where $a(\omega\cdot t,\cdot)=\omega_a\cdot t$, $b(\omega\cdot t,\cdot)=\omega_b\cdot t$.

For any $v\in X$, let $\psi(t,x;v, \omega)$ be the solution  of \eqref{linear-equa-induced} with $\psi(0,x;v,\omega)=v(x),x\in S^1$; and write
 $$\Psi(t,\omega):X\rightarrow X;v(\cdot)\mapsto \Psi(t,\omega)v:=\psi(t,\cdot;v, \omega)$$ as the evolution operator generated by  \eqref{linear-equa-induced}. Then we have the following

\begin{lem}\label{bounded-operator}
There exists $M>0$ (depending only on $\alpha$, $\|a\|_{\infty}$ and $\|b\|_{\infty}$) such that $\|\Psi(t,\omega)\|<M$, for all $0\leq t\leq 1$ and $\omega\in\Omega$.
\end{lem}

\begin{proof}
Let $A_0v=v_{xx}$ and $A(\omega\cdot t)v=a(\omega\cdot t,\cdot)v_x+b(\omega\cdot t,\cdot)v$.  Then  for any $v_0\in X$ and $t\ge 0$,
$$
\Psi(t,\omega)v_0=e^{A_0 t}v_0+\int_0^t e^{A_0(t-s)} A(\omega\cdot s) \Psi(s,\omega)v_0 ds.
$$
Note that
$$
\|e^{A_0t}v_0\|=\|e^{A_0t}A_0^\alpha v_0\|_{L^2}\le \|A_0^\alpha v_0\|_{L^2}=\|v_0\|.
$$
Hence
\begin{align*}
\|\Psi(t,\omega)v_0\| &\le \|v_0\|+\int_0^ t \|e^{A_0(t-s)}A(\omega\cdot s) \Psi(s,\omega)v_0\| ds\\
& = \|v_0\|+\int_0^ t \|A^{\alpha}_0e^{A_0(t-s)} A(\omega\cdot s) \Psi(s,\omega)v_0\|_{L^2}ds.\\
& \le \|v_0\| + \int_0^ t\frac{1}{(t-s)^\alpha}\|\Psi(s,\omega)v_0\|_{L^2}ds\\
& \le \|v_0\| + (\|a\|_{\infty}+\|b\|_{\infty})\int_0^ t\frac{1}{(t-s)^\alpha}\|\Psi(s,\omega)v_0\|ds.
\end{align*}
It then follows from \cite[Lemma 7.1.1 or Exercise 4*, page 190]{Hen} that $\|\Psi(t,\omega)\|\leq M(\alpha,\|a\|_{\infty},\|b\|_{\infty})$.
\end{proof}

Therefore, $\Psi(t,\omega)\in \mathcal{L}(X)$, where $\mathcal{L}(X)$ denotes the linear bounded operator on $X$. Moreover, for any $t>0$ and $\omega\in \Omega$, $\Psi(t,\omega)$ is injective, and by standard a priori estimates for parabolic equations, $\Psi(t,\omega)$ is a compact operator from $X$ to $X$.
\medskip

Put $\Pi^t:X\times \Omega\to X\times \Omega$ as
\begin{equation}\label{linear-equa-semi}
\Pi^t(v,\omega)=(\Psi(t,\omega)v,\omega\cdot t)\,\text{ for } t\ge 0.
\end{equation}
Then the family of  maps $\{\Pi^t\}_{t\ge 0}$ so defined is a linear skew-product semiflow on $X\times\Omega$ (see, e.g. \cite[p.57]{Mierczynski}).
We say $\Psi$ admits an {\it exponential dichotomy over} (ED) $\Omega$ if there exist $K>0$, $\beta>0$ and projections $P(\omega):X\rightarrow X$ {(here, $P(\omega)$ varies continuously in $\omega$)} such that for all $\omega \in \Omega$, $\Psi(t,\omega)|_{R(P(\omega))}:R(P(\omega))\rightarrow R(P(\omega\cdot t))$ is an isomorphism satisfying $\Psi(t,\omega)P(\omega)=P(\omega \cdot t)\Psi(t,\omega)$, $t\in \mathbb{R}^+$; moreover,
\begin{equation}\label{exponential-dicho-express}
\begin{split}
\|\Psi(t,\omega)(I-P(\omega))\|\leq Ke^{-\beta t},\quad t\geq 0,\\
\|\Psi(t,\omega)P(\omega)\|\leq Ke^{\beta t},\quad t\leq 0.
\end{split}
\end{equation}
Here $R(P(\omega))$ is the range of $P(\omega)$.
\medskip

Now, consider the perturbation equation of \eqref{linear-equa}:
\begin{equation}\label{linear-equa-pertur}
v_t= v_{xx}+(a(t,x)+\epsilon^a(t,x))v_x+(b(t,x)+\epsilon^b(t,x))v,\quad x\in S^1,
\end{equation}
where $\epsilon^b(t,x)$ is bounded and uniformly continuous on $\mathbb{R}\times S^1$, $\epsilon^a(t,x)\in C^1(\mathbb{R}\times S^1)$, $\epsilon^a_t(t,x)$ and $\epsilon^a_x(t,x)$ are bounded.

We define
$$
H(a, {\epsilon}^a, b,{\epsilon}^b)={\rm cl}\{(a\cdot \tau,{\epsilon}^a\cdot \tau,b\cdot\tau,{\epsilon}^b\cdot \tau)=(a(t+\tau,\cdot),\epsilon^a(t+\tau,\cdot),b(t+\tau,\cdot),\epsilon^b(t+\tau,\cdot)):\tau\in \RR\}.
 $$
 Let
 $$
 \Omega^{\epsilon}=H(a,{\epsilon}^a,b,{\epsilon}^b)
  $$
  and denotes the element in $\Omega^{\epsilon}$ by $\omega^{\epsilon}=(\omega_a,\omega_a^{\epsilon},\omega_b,\omega_b^{\epsilon})$. Then, equation \eqref{linear-equa-pertur} can also give rise to a family of equations,
\begin{equation}\label{linear-equa-perturinduced}
v_t= v_{xx}+a^{\epsilon}(\omega^{\epsilon}\cdot t,x)v_x+b^{\epsilon}(\omega^{\epsilon}\cdot t,x)v,\,\,t>0,\,x\in S^{1}=\mathbb{R}/2\pi \mathbb{Z},
\end{equation}
where $a^{\epsilon}(\omega^{\epsilon}\cdot t,\cdot)=\omega_a\cdot t +\omega_a^{\epsilon}\cdot t$, and $b^{\epsilon}(\omega^{\epsilon}\cdot t,\cdot)=\omega_b\cdot t+ \omega_b^{\epsilon}\cdot t$.

Let $\Psi^{\epsilon}(t,\omega^{\epsilon})$ be the evolution operator generated by \eqref{linear-equa-perturinduced}.
Then \eqref{linear-equa-perturinduced} also induces a linear skew-product semiflow $\Pi_{\epsilon}^{t}$ on $X\times \Omega^{\epsilon}$:
\begin{equation}\label{linear-equa-pertursemi}
\Pi_{\epsilon}^t(v,\omega)=(\Psi^{\epsilon}(t,\omega^{\epsilon})v,\omega^{\epsilon}\cdot t)\,\text{ for } t\ge 0.
\end{equation}

Observe that the compact flow $(\Omega^\epsilon, \cdot)$ is an extension of the compact flow $(\Omega,\cdot)$. To be more precise,
let $P^\epsilon: \Omega^\epsilon\to \Omega$ be defined as follows:
 \begin{equation}
 \label{projection-eq}
 P^\epsilon(\omega^\epsilon)=(\omega_a,\omega_b)\quad \forall \,\, \omega^\epsilon=(\omega_a,\omega_a^\epsilon,\omega_b,\omega_b^\epsilon).
 \end{equation}
 Then
 \begin{equation}
 \label{extension-flow-eq}
 P^\epsilon(\omega^\epsilon\cdot t)=P^\epsilon(\omega^\epsilon)\cdot t\quad \forall\,\, \omega^\epsilon\in\Omega^\epsilon,\,\, t\in\mathbb{R}.
 \end{equation}
 Hence  the compact flow $(\Omega^\epsilon, \cdot)$ is an extension of the compact flow $(\Omega,\cdot)$.
 Hereafter,  for any $\omega^{\epsilon}\in \Omega^{\epsilon}$, we always put $\omega=P^\epsilon(\omega^\epsilon)$.

The following two lemmas generalize the results in \cite{Chow-Leiva95} related to the persistence of exponential dichotomy of linear skew-product semiflow under small perturbations on a compact base flow to linear skew-product semiflow under small perturbations on two different compact base flows.

\begin{lem}\label{perturbation-bounded-operator}
There is $M_{\epsilon}>0$ (only depending on $\|a\|_\infty,\|\epsilon^a\|_\infty, \|b\|_\infty,\|\epsilon^b\|_\infty$) such that
$$
\sup_{0\le t\le 1}\|\Psi^{\epsilon}(t,\omega^{\epsilon})-\Psi(t,\omega)\|\le M_{\epsilon}\big(\|\epsilon^a\|_{\infty}+\|\epsilon^b\|_{\infty}\big),
$$
where $\omega=P^\epsilon(\omega^\epsilon)$.
\end{lem}

\begin{proof}
By Lemma \ref{bounded-operator}, there is $M'_{\epsilon}$ (only depending on $\|a\|_\infty,\|\epsilon^a\|_\infty, \|b\|_\infty,\|\epsilon^b\|_\infty$) such that $\|\Psi^{\epsilon}(t,\omega^{\epsilon})\|\leq M'_{\epsilon}$, for $0\leq t\leq 1$ and $\omega^{\epsilon}\in \Omega^{\epsilon}$.

Let $A(\omega^{\epsilon}\cdot t)v=a^{\epsilon}(\omega^{\epsilon}\cdot t,x)v_x+b^{\epsilon}(\omega^{\epsilon}\cdot t,x)v$. Then
$$
\Psi^{\epsilon}(t,\omega^{\epsilon})v_0=\Psi(t,\omega)v_0+\int_0^t \Psi(t-s,\omega\cdot s)\Big(A(\omega^{\epsilon}\cdot s)-A(\omega\cdot s)\Big)\Psi^{\epsilon}(s,\omega^{\epsilon})v_0 ds,
$$
where $A(\omega \cdot s)$ is as defined in Lemma \ref{bounded-operator}. Hence
\begin{align*}
\|\Psi^{\epsilon}(t,\omega^{\epsilon})v_0-\Psi(t,\omega)v_0\| &\le \int_0^t \|\Psi(t-s,\omega\cdot s)\Big(A(\omega^{\epsilon}\cdot s)-A(\omega\cdot s)\Big)\Psi^{\epsilon}(s,\omega^{\epsilon})v_0 \|ds\\
&\le M \int_0^t\frac{1}{(t-s)^\alpha} \| \Big(A(\omega^{\epsilon}\cdot s)-A(\omega\cdot s)\Big)\Psi^{\epsilon}(s,\omega^{\epsilon})v_0\|_{L^2}ds\\
&\le M \big(\|\epsilon^a\|_{\infty}+\|\epsilon^b\|_{\infty}\big)\int_0^t\frac{1}{(t-s)^\alpha}\|\Psi^{\epsilon}(s,\omega^{\epsilon})v_0\|ds\\
&\le  \frac{MM'_{\epsilon} \big(\|\epsilon^a\|_{\infty}+\|\epsilon^b\|_{\infty}\big)}{1-\alpha}\|v_0\|.
\end{align*}
Let $M_{\epsilon}=\frac{MM'_{\epsilon}}{1-\alpha}$, the lemma then follows.
\end{proof}

\begin{lem}\label{pertubation-invariant}
 Assume that $\Pi^t$ admits an exponential dichotomy over $\Omega$. {Then}, there exists $\delta>0$ such that if  $|\epsilon^a(t,x)|<\delta$, $|\epsilon^a_t(t,x)|<\delta$, $|\epsilon^a_x(t,x)|<\delta$ and $|\epsilon^b(t,x)|<\delta$ for any $(t,x)\in \mathbb{R}\times S^1$,
  then $\Pi^t_{\epsilon}$ admits an exponential dichotomy over $\Omega^{\epsilon}$.
\end{lem}

\begin{proof}
Since $\Pi^t$ admits an exponential dichotomy over $\Omega$, there exist $\beta, K>0$ such that
\begin{equation}\label{exponential-dicho-express2}
\begin{split}
\|\Psi(t,\omega)(I-P(\omega))\|\leq Ke^{-\beta t},\quad t\geq 0,\\
\|\Psi(t,\omega)P(\omega)\|\leq Ke^{\beta t},\quad t\leq 0,
\end{split}
\end{equation}
where $P(\cdot)$ is the associated projection.

Then the mapping $\hat\Pi:X\times \Omega\times \mathbb{Z}\to X\times \Omega$ given by
\begin{equation}\label{skew-product-sequen}
\hat{\Pi}(v, \omega, n) :=\left(\Psi_n(\omega)v, \omega \cdot n\right)
\end{equation}
where
$$
\Psi_{\mathrm{n}}(\omega)=\Psi(1,\omega \cdot n)
$$
is a skew-product sequence, see \cite[Definition 3.1 and Proposition 4.1]{Chow-Leiva95}. Therefore, the skew-product sequence $\hat\Pi$ given by \eqref{skew-product-sequen} admits a uniform discrete dichotomy over $\Omega$  (see \cite[Definition 3.3]{Chow-Leiva95} for the definition
of uniform discrete dichotomy)
as follows,
\begin{equation*}
\begin{split}
\|\Psi_{n,m}(\omega)(I-P(\omega\cdot n))\|\leq K\eta^{m-n},\quad n\leq m\in \mathbb{Z},\\
\|\Psi_{n,m}(\omega)P(\omega\cdot n)\|\leq K\eta^{-(m-n)},\quad n>m \in \mathbb{Z},
\end{split}
\end{equation*}
where $\Psi_{n,m}(\omega)=\Psi_{m-1}\Psi_{m-2}\cdots \Psi_{n}(\omega)$ if $m>n$, and $\Psi_{n,m}(\omega)=\Psi_{m}^{-1}\Psi^{-1}_{m+1}\cdots \Psi^{-1}_{n-1}(\omega)$ if $n>m$, with $\eta=e^{-\beta}$.

Moreover, by \cite[Theorem 7.6.7]{Hen}, for any given $K_1>K$ and $\eta<\eta_1<1$, there exists $\delta_0>0$ (depending only on $K$, $K_1$, $\eta$ and $\eta_1$) such that for any $\omega\in\Omega$ and any sequence $\{\Phi_n\}_{-\infty}^{\infty}\subset \mathcal{L}(X)$ with $\sup_{n\in \mathbb{Z}}\|\Psi_n(\omega)-\Phi_n\|\leq\delta_0$ has discrete dichotomy with $M_1$ and $\eta_1$.

Similarly, the mapping $\hat\Pi_{\epsilon}:X\times \Omega^{\epsilon}\times \mathbb{Z}\to X\times \Omega^{\epsilon}$ given by
\begin{equation}\label{skew-product-sequen-per}
\hat\Pi_{\epsilon}(v, \omega, n) :=\left(\Psi^{\epsilon}_n(\omega^{\epsilon})v, \omega^{\epsilon} \cdot n\right)
\end{equation}
is also a skew-product sequence,
where $\Psi^{\epsilon}_{\mathrm{n}}(\omega^{\epsilon})=\Psi(1,\omega^{\epsilon} \cdot n)$.

It  follows from Lemma \ref{perturbation-bounded-operator} that $\|\Psi^{\epsilon}_{\mathrm{n}}(\omega^{\epsilon})-\Psi_{\mathrm{n}}(\omega)\|\leq 2M_{\epsilon}\delta$ for any $n\in\mathbb{Z}$. Now choose $\delta>0$ sufficiently small such that $2M_{\epsilon}\delta<\delta_0$. Then, for each $\omega^{\epsilon}\in \Omega^{\epsilon}$, the skew-product sequence $\hat\Pi_{\epsilon}$ given by \eqref{skew-product-sequen-per} admits a
 discrete dichotomy with $M_1$ and $\eta_1$. This implies that $\hat\Pi_{\epsilon}$ admits a uniform discrete dichotomy over $\Omega^{\epsilon}$, here the dimension of the associated projected spaces are independent of $\omega^{\epsilon}$. Therefore, $\Pi_{\epsilon}$ is weakly hyperbolic on $\Omega^{\epsilon}$ (See \cite[p.27]{Sacker1991} for the definition of weakly hyperbolic). By \cite[Theorem F]{Sacker1991}, $\Pi_{\epsilon}$ admits an exponential dichotomy on $\Omega^{\epsilon}$ and hence $P(\omega^{\epsilon})$ continuously depend on $\omega^{\epsilon}$.
The proof of this lemma is completed.
\end{proof}

\subsection{Sacker-Sell spectrum and associated invariant spaces}
Let $\lambda\in \mathbb{R}$ and define $L^t_{\lambda}: X\times\Omega\rightarrow X\times\Omega$ by
\begin{equation}\label{exponential-dicho}
  L^t_{\lambda}(v,\omega)=(\Psi_{\lambda}(t,\omega)v,\omega\cdot t),
\end{equation}
where $\Psi_{\lambda}(t,\omega)=e^{-\lambda t}\Psi(t,\omega)$. We call
$$
\sigma(\Omega)=\{\lambda\in \mathbb{R}\, |\, L^t_{\lambda}\ \text{has no exponential dichotomy over }\Omega\}
$$
the {\it Sacker-Sell spectrum} of \eqref{linear-equa-induced}. Recall that $\Omega$ is compact and connected, the Sacker-Sell spectrum $\sigma(\Omega)=\bigcup_{k=0}^\infty I_k$, where $I_k=[a_k,b_k]$ and $\{I_k\}$ is ordered from right to left, that is, $\cdots<a_k\leq b_k<a_{k-1}\leq b_{k-1}<\cdots<a_0\leq b_0$ (see \cite{Chow1994, ChLuMa, Sacker1978,Sacker1991}).

For any given $0\leq n_1\leq n_2\leq\infty$, if $n_2\neq \infty$, let
\begin{equation}\label{twoside-estimate}
\begin{split}
E^{n_1,n_2}(\omega)=\{v\in X\, |\, &\|\Psi(t,\omega)v\|=o(e^{a^-t})\ \text{as}\ t\rightarrow -\infty,\\ & \|\Psi(t,\omega)v\|=o(e^{b^+t})\ \text{as}\ t\rightarrow \infty\}
\end{split}
\end{equation}
where $a^-$, $b^+$ are such that $\lambda_1<a^-<a_{n_2}\leq b_{n_1}<b^+<\lambda_2$ for any $\lambda_1\in \cup_{k=n_2+1}^{\infty}I_k$ and $\lambda_2\in\cup_{k=0}^{n_1-1}I_k$. If $n_2=\infty$, let
\begin{equation}\label{infty-estimate}
E^{n_1,\infty}(\omega)=\{v\in X :\|\Psi(t,\omega)v\|=o(e^{b^+t})\text{ as }t\to\infty\}
\end{equation}
where $b^+$ is such that $b_{n_1}<b^+<\lambda$ for any $\lambda\in \cup_{k=0}^{n_1-1}I_k$.

$E^{n_1,n_2}(\omega)$ is invariant in the sense: for $t\geq 0$, $\Psi(t,\omega)E^{n_1,n_2}(\omega)=E^{n_1,n_2}(\omega\cdot t)$ when $n_2<\infty$, while $E^{n_1,\infty}(\omega)$ satisfies $\Psi(t,\omega)E^{n_1,\infty}(\omega)\subset E^{n_1,\infty}(\omega\cdot t)$
(see \cite[Remark 2.3(ii)]{SWZ}).

{Assume that $0\not\in \sigma(\Omega)$ and $n_0$ is such that $I_{n_0}\subset (0,\infty)$ and $I_{n_0+1}\subset (-\infty,0)$.
$E^s(\omega)=E^{n_0+1,\infty}(\omega)$ and $E^u(\omega)=E^{0,n_0}(\omega)$ denote the {\it stable} and {\it unstable subspaces} of \eqref{linear-equa-induced} at $\omega\in \Omega$, respectively. In this case, $\Omega$ is called hyperbolic.}

 Assume $0\in \sigma(\Omega)$ and $n_0$ be such that $0\in I_{n_0}\subset\sigma(\Omega)$. Then $E^s(\omega)=E^{n_0+1,\infty}(\omega)$, $E^{cs}(\omega)=E^{n_0,\infty}(\omega)$, $E^{c}(\omega)=E^{n_0,n_0}(\omega)$, $E^{cu}(\omega)=E^{0,n_0}(\omega)$, and $E^u(\omega)=E^{0,n_0-1}(\omega)$ denote the {\it stable, center stable, center, center unstable}, and {\it unstable subspaces} of \eqref{linear-equa-induced} at $\omega\in \Omega$, respectively.

We have the following zero number properties on these invariant spaces from \cite[Lemma 2.7]{SWZ}.
\begin{lem}\label{zero-inva}
For given $0\leq n_1\leq n_2\leq\infty$ (when $n_2=\infty$, $n_1<\infty$ is needed), we have
\begin{equation*}
N_1\leq Z(v(\cdot))\leq N_2,\,\,\text{ for any }v\in E^{n_1,n_2}(\omega)\setminus\{0\},
\end{equation*}
where
\begin{equation*}
N_1=\left\{
\begin{split}
 &{\rm dim}E^{0,n_1-1}(\omega),\,\quad\,\,\,\text{ if }{\rm dim}E^{0,n_1-1}(\omega)\text{ is even;}\\
 &{\rm dim}E^{0,n_1-1}(\omega)+1,\,\text{ if }{\rm dim}E^{0,n_1-1}(\omega)\text{ is odd,}
\end{split}\right.
\end{equation*} and
\begin{equation*}
N_2=\left\{
\begin{split}
 &{\rm dim}E^{0,n_2}(\omega),\,\quad\,\,\,\text{ if }{\rm dim}E^{0,n_2}(\omega)\text{ is even;}\\
 &{\rm dim}E^{0,n_2}(\omega)-1,\,\text{ if }{\rm dim}E^{0,n_2}(\omega)\text{ is odd.}
\end{split}\right.
\end{equation*}
Here, we define $E^{0,-1}(\omega)=\{0\}.$
\end{lem}

\subsection{Invariant manifolds for nonlinear parabolic equations on $S^1$}
Given any $u\in X$ and $a\in S^1$, we define the shift $\sigma_a$ on $u$ as $(\sigma_a u)(\cdot)=u(\cdot+a)$. Then the {\it $S^1$-group orbit} of $u$ is defined as the following:
 \begin{equation}\label{E:group-orbit-11}
 \Sigma u=\{\sigma_a u\,|\, a\in S^1\}.
 \end{equation}
Let $u(t,\cdot;u_0)$ be a classical solution of \eqref{general-equa}+\eqref{autonomous} (resp. \eqref{general-equa}+\eqref{periodic-eq}) with $u(0,\cdot;u_0)=u_0\in X$, then $\sigma_au(t,\cdot;u_0)$ is a classical solution of \eqref{general-equa}+\eqref{autonomous} (resp. \eqref{general-equa}+\eqref{periodic-eq}). Moreover, the uniqueness of solution ensures the {\it translation invariance}, that is, $\sigma_au(t,\cdot;u_0)=u(t,\cdot;\sigma_au_0)$.

 Assume moreover that $u(t,x;u_0)$ is a bounded solution of \eqref{general-equa}+\eqref{autonomous} (resp. \eqref{general-equa}+\eqref{periodic-eq}) in $X$ and $\omega(u_0)$  (resp. $\tilde \omega(u_0)$) is the $\omega$-limit set of $u_0$ (resp. with respect to the time $T$-map or Poincar\'{e} map $P$). Then, by \cite[Theorem A]{Matano}(resp. by \cite[Theorem 1]{SF1})
$$
\omega(u_0)\subset \Sigma \phi\ (\text{resp. }\tilde \omega(u_0)\subset \Sigma \phi),
$$
where $\phi$ is a fixed point of { \eqref{general-equa}+\eqref{autonomous}} or $u(t,x;\phi)=\phi(x-ct)$ for some $c\not =0$, which is referred to as a rotating wave (resp. $\phi\in X$ and there is $r\in S^1$ such that $P \phi=\sigma_r \phi$).

\begin{lem}
{ Let  $\tilde u\in \omega(u_0)$ (resp. $\tilde u\in \tilde \omega(u_0)$). Assume that $\mathcal{O}=\mathrm{cl}\{u(t,\cdot;\tilde u), t\in\mathbb{R}\}$ is hyperbolic.}  Then
$\phi$ (resp. $\phi$ with respect to $P$) is a spatially-homogeneous fixed point.
\end{lem}
\begin{proof}
We only prove this lemma in the autonomous case, while the deduction for periodic case is analogous.

Since $\mathcal {O}\subset \omega(u_0)\subset \Sigma \phi$, it suffices to prove that $\mathcal {O}$ is spatially-homogeneous (i.e. all the elements in $\mathcal {O}$ are spatially-homogeneous). The hyperbolicity of $\mathcal {O}$ implies the following equation:
\begin{equation}\label{linearized-equation}
  v_t=v_{xx}+\p_2 f(u(t,\cdot;\tilde u),u_x(t,\cdot;\tilde u)) v_x+\p_1 f(u(t,\cdot;\tilde u),u_x(t,\cdot;\tilde u)) v
\end{equation}
admits an exponential dichotomy.

Suppose on the contrary that $\tilde u_0\in \mathcal{O}$ is spatially-inhomogeneous. Let $v(t,x)=u_x(t,x;\tilde u_0)$, then it is not hard to see that $v(t,x)$ is a solution of \eqref{linearized-equation}. Since $\mathcal {O}\subset \Sigma \phi$, $v(t,x)\in \Sigma \phi_x$, and then $\|v(t,x)\|=\|\phi_x\|\neq 0$.

Since $\mathcal{O}$ is hyperbolic, $v(t,x)=v^u(t,x)+v^s(t,x)$ with $v^u(t,x)\in E^u(\p_2 f(u(t,\cdot;\tilde u_0),u_x(t,\cdot;\tilde u_0)),\\ \p_1 f(u(t,\cdot;\tilde u_0),u_x(t,\cdot;\tilde u_0))\setminus\{0\}$ and $v^s(t,x)\in E^s(\p_2 f(u(t,\cdot;\tilde u_0),\p_1 f(u(t,\cdot;\tilde u_0))\setminus\{0\}$, for otherwise it will contradict to that $\|v(t,\cdot)\|=\|\phi_x\|$ for all $t\in\mathbb{R}$. The fact that $\|v^s(t,\cdot)\|$ is decreasing exponentially and $\|v^u(t,\cdot)\|$ is increasing exponentially as $t\to \infty$ imply that $\|v(t,\cdot)\|$ is unbounded, a contradiction. Thus, $\mathcal{O}$ is spatially-homogeneous. We have completed the proof of this lemma.
\end{proof}

Therefore, we naturally have the following definitions:
\begin{defn}
{\rm
Assume that $\phi$ is a hyperbolic equilibrium of { \eqref{general-equa}+\eqref{autonomous}}, $U\subset X$ is an open neighborhood of $\phi$. Then the local stable manifold $W_{loc}^s(\phi)$ and local unstable manifold $W_{loc}^u(\phi)$ of $\phi$ are defined as follows:
\begin{equation*}
\begin{split}
W_{loc}^{s}(\phi)=&\left\{\hat u \in U: u(t,\cdot;\hat u) \in U \text { for all } t \geq 0 \text { and } \|u(t,\cdot;\hat u)-\phi\|\rightarrow 0 \text { exponentially as } t \rightarrow \infty\right\}\\
W_{loc}^{u}(\phi)=&\left\{\hat u \in U: \text { some backward branch } u(t,\cdot;\hat u) \text { exists for all } t<0 \text { and lies in } U,\right.\\ & \text { further, } \|u(t,\cdot;\hat u)-\phi\| \rightarrow 0 \text { exponentially as } t \rightarrow-\infty \}.
\end{split}
\end{equation*}

}
\end{defn}

\begin{defn}
{\rm
Let $\phi(t)$ be a hyperbolic $T$-periodic orbit of { \eqref{general-equa}+\eqref{periodic-eq}} and for any $\tau\in[0,T)$ assume $U_{\tau}\subset X$ be an open neighborhood of $\phi(\tau)$, then the local stable manifold $W_{loc}^s(\phi(\tau))$ and local unstable manifold $W_{loc}^u(\phi(\tau))$ of $\phi(\tau)$ are defined as follows:
\begin{equation*}
\begin{split}
W_{loc}^{s}(\phi(\tau))=&\left\{\hat u \in U_{\tau}: u(nT+\tau,\cdot;\hat u) \in U_{\tau} \text { for all } n \in \mathbb{N} \text { and } \|u(nT+\tau,\cdot;u_0)-\phi(\tau)\|\rightarrow 0 \right. \\ & \text { exponentially as } n \rightarrow \infty\}.\\
W_{loc}^{u}(\phi(\tau))=&\left\{\hat u \in U_{\tau}: \text { some backward branch } u(t,\cdot;\hat u) \text { exists for all } t<0 \text { and } u(-nT+\tau,\cdot;\hat u)\right. \\ &  \text{lies in } U_{\tau}, \text { further, } \|u(-nT+\tau,\cdot;\hat u)-\phi(\tau)\| \rightarrow 0 \text { exponentially as } n \rightarrow \infty \}.
\end{split}
\end{equation*}

}
\end{defn}

As the end of this section, we recall some properties concerning with local unstable manifolds of hyperbolic fixed points (resp. hyperbolic fixed points of the associated Poincar\'{e} map $P$) of \eqref{general-equa}+\eqref{autonomous} (resp. \eqref{general-equa}+\eqref{periodic-eq}) which will be used in later sections, as the following

\begin{lem}\label{local-um}
  Suppose that $\phi$ is a hyperbolic fixed point (resp. hyperbolic fixed point of the associated Poincar\'{e} map $P$) of \eqref{general-equa}+\eqref{autonomous} (resp. \eqref{general-equa}+\eqref{periodic-eq}). Then there is an open neighborhood $U\subset X$ of $\phi$ such that for any $\hat u\in X$ satisfies $u(t,\cdot;\hat u)\in \overline U$ (resp. $u(-nT,\cdot;\hat u)\in \overline U$) for $t\leq 0$ (resp. $n\in \mathbb{N}$), one has $\hat u\in W^u_{loc}(\phi)$.
\end{lem}
\begin{proof}
  See \cite[Sec.6.3]{Ruelle89}.
\end{proof}

\section{The autonomous case}

{ In this section, the solution operator of \eqref{general-equa}+\eqref{autonomous} is defined as follows:
$$
S_tu_0=u(t,\cdot;u_0),\quad t\geq 0,\, u_0\in X.
$$
Assume that for each $u_0\in X$, $u(t,\cdot;u_0)$ exists for all $t\ge 0$.
Then, $S_t$ defined as above induces a semiflow on $X$.
A compact invariant set $A$ of the semiflow $S_t$ is said to be a {\it maximal compact invariant set} if every compact invariant set of $S_t$
 is a subset of $A$. An invariant set $\mathcal{A}$ is said to be a {\it global attractor} {(see, e.g. \cite[p.39]{Hale88})} if $\mathcal{A}$ is a maximal compact invariant set which attracts each bounded set $B\subset X$, that is
$$
\lim_{t\to\infty}\sup_{u'\in B}\inf_{u\in\mathcal{A}}\|S_tu'-u\|=0.
$$
}

Hereafter, {we always assume that  for each $u_0\in X$, $u(t,\cdot;u_0)$ exists for all $t\ge 0$ and that \eqref{general-equa}+\eqref{autonomous} admits a compact global attractor $\mathcal{A}$.}

\begin{defn}\label{non-wan-auto}
{\rm A point $u_0\in X$ is said to be a non-wandering point of \eqref{general-equa}+\eqref{autonomous} if for any neighborhood
 $U$ of $u_0$, {and any $T_0>0$, there exists $t>T_0$ such that $\{u(t,\cdot;u')\,|\, u'\in U\}\cap U\not =\emptyset$ (see, e.g. \cite[p.106]{wiggins}).} }
  \end{defn}

 Observe that if $u_0\in X$ is a non-wandering point, then for any $t_n\to\infty$, there are $t_n^{'}>t_n$ and $u_n\in X$ such that
   $u_n\to u_0$ and $u(t_n^{'},\cdot;u_n)\to u_0$ as $n\to\infty$.
   It then follows that $u(t,x;u_0)$ exists for all $t\in \RR$ and $\{u(t,\cdot;u_0)\,|\, t\in\RR\}\subset \mathcal{A}$. In fact, suppose that $t_n^{'}\to\infty$ and $u_n\in X$ are such that
   $u_n\to u_0$ and $u(t_n^{'},\cdot;u_n)\to u_0$ as $n\to\infty$. Noticing that $\{u(\cdot+t_n^{'},\cdot;u_n)\}$ is precompact in {$C(I\times X)$} for any bounded closed interval $I\subset \R$, without loss of generality, {for any given $t\in\R$}, we may assume that
   $u(t+t_n^{'},\cdot;u_n)\to \tilde u(t)$ as $n\to\infty$. This together with $u(t_n^{'},\cdot;u_n)\to u_0$ implies that
   $u(t,\cdot;u_0)$ exists and $u(t,\cdot;u_0)=\tilde u(t)=\lim_{n\to\infty}u(t+t_n^{'},\cdot;u_n)$ for {any given $t\in\R$}.

   \smallskip

In the rest of this section, we always assume that $u_0$ is a non-wandering point of \eqref{general-equa}+\eqref{autonomous}. Let $\omega(u_0)$ be the $\omega$-limit set of $u_0$. By \cite[Theorem A]{Matano},
\begin{equation}
\label{omega-limit-set-eq}
\omega(u_0)\subset \{\sigma_a \phi\,|\, a\in S^1\}\quad \text{for some }\phi\in X.
\end{equation}
Moreover, $\phi$ is a fixed point of \eqref{general-equa} or $u(t,x;\phi)=\phi(x-ct)$ for some $c\not =0$. Hereafter, $\phi$ is assumed to be  as in \eqref{omega-limit-set-eq}.

 We say that $u_0$ {\it generates a rotating wave} if $u_0(\cdot)\not\equiv {\rm const}$ and
$u(t,x;u_0)=u_0(x-ct)$ for some $c\not =0$.
A point $u\in X$ is called {\it spatially-homogeneous} if $u(\cdot)$ is independent of the spatial variable $x$. Otherwise, $u$ is called {\it spatially-inhomogeneous}.
\vskip 2mm

Our main result in this section is the following
\begin{thm}
\label{aut-thm} Suppose  that $u_0$ is a non-wandering point of \eqref{general-equa}+\eqref{autonomous}. Then
 $u_0$ is a fixed point or generates a rotating wave. Particularly, if $f(u,u_x)=f(u,-u_x)$, then $u_0$ is a fixed point.
  \end{thm}

Before proving the above theorem,  we first prove the following lemmas under the assumption that $u_0$ is not a fixed point and does not generate a rotating wave.

Note that if  $\phi$ is spatially-homogeneous, then $\phi$ is a fixed point of  \eqref{general-equa}+\eqref{autonomous}. In such case, the following is the linearization of  \eqref{general-equa} at $\phi$,
\begin{equation}\label{fixed-point-linearized}
v_t=v_{xx}+a v_x+b v,\quad x\in S^1,
\end{equation}
where $a=\p_2 f(\phi,0)$ and $b=\p_1 f(\phi,0)$. We denote $E^u=E^u(a,b)$, $E^c=E^c(a,b)$, and $E^{cu}=E^{cu}(a,b)$ as the unstable, center, and center unstable spaces of \eqref{fixed-point-linearized}, respectively.

\begin{lem}\label{fixed-constant}
Assume that $u_0$ is neither a fixed point nor generates a rotating wave.
\begin{itemize}
\item[{\rm (i)}]
If $\phi$ is spatially-inhomogeneous, then there is $N\in\NN$ such that
$$
Z(\phi-\sigma_a \phi)=N,\quad \forall a\in S^1\setminus\{0\}.
$$
\item[{\rm (ii)}]If $\phi$ is spatially-homogeneous, then $\phi$ is neither stable (which implies $\dim E^{cu}>0$) nor hyperbolic.  And either $\dim E^u=0$ with $\dim E^c=1$ or $\dim E^u$ is odd with $\dim E^c=2$.
Moreover, there is $N_0\in \NN$ such that
$$
Z(v(\cdot))=N_0, \quad \forall v\in E^c\setminus\{0\}.
$$
\end{itemize}
\end{lem}

\begin{proof}

(i) Note that either both $\phi$ and $\sigma_a\phi$ are fixed points of \eqref{general-equa}+\eqref{autonomous} or $u(t,x;\phi)$ and $u(t,x;\sigma_a\phi)$ are periodic solutions
of \eqref{general-equa}+\eqref{autonomous}. It then follows from Lemma \ref{zero-number} that
$Z(u(t,\cdot;\phi)-u(t,\cdot;\sigma_a \phi))$ is independent of $t$, and hence $u(t,\cdot;\phi)- u(t,\cdot;\sigma_a \phi)$ has only simple zeros and $Z(u(t,\cdot;\phi)-u(t,\cdot;\sigma_a \phi))$ is continuous in $a\in S^1\setminus\{0\}$. Therefore, there is $N\in\NN$ such that (i) holds.

\smallskip

(ii) {Suppose on the contrary that $\phi$ is stable. Then, for the given $0<\epsilon<\|u_0-\phi\|$, there is $\eta>0$ such that $\|u(t,\cdot;\tilde \phi)-\phi\|<\dfrac{\epsilon}{2}$ for $t>0$ and $\tilde \phi\in X$ with $\|\tilde\phi-\phi\|<\eta$. Noticing that $u(t,\cdot;u_0)\to \phi$ as $t\to\infty$, there exists $T_0>0$ satisfying $\|u(T_0,\cdot;u_0)-\phi\|<\eta/2$. Hence, $\|u(t,\cdot;u_0)-\phi\|<\dfrac{\epsilon}{2}$ for all $t>T_0$. Since $u_0$ is a non-wandering point, there are $u_n \in X$ and $t^{'}_n\to \infty$ with $u_n \to u_0$ and $u(t^{'}_n,\cdot;u_n )\to u_0$, as $n\to \infty$. In particular, one can chose $N_1>0$ be such that $\|u(T_0,\cdot;u_n )-u(T_0,\cdot;u_0)\|<\dfrac{\eta}{2}$ and $t^{'}_n>T_0$ for $n>N_1$. Therefore,  $\|u(T_0,\cdot;u_n )-\phi\|\leq \|u(T_0,\cdot;u_n )-u(T_0,\cdot;u_0)\|+\|u(T_0,\cdot;u_0)-\phi\|<\eta$; moreover, $\|u(t^{'}_n,\cdot;u_n )-\phi\|<\dfrac{\epsilon}{2}$ for $n>N_1$. As a consequence, $\|u(t^{'}_n,\cdot;u_n )-u_0\|\geq \|\phi-u_0\|-\|u(t^{'}_n,\cdot;u_n )-\phi\|\geq \|\phi-u_0\|-\dfrac{\epsilon}{2}\geq \dfrac{\epsilon}{2}$ for all $n>N_1$, a contradiction to $u(t^{'}_n,\cdot;u_n )\to u_0$. That is, $\phi$ is not stable and $\dim E^{cu}>0$.

We now turn to prove that $\phi$ cannot be hyperbolic. For simplicity, let $t_n^{'}\nearrow \infty$ be such that $u(t_n^{'},\cdot;u_n )\to u_0$ as $n\to\infty$ and choose $t_n<t_n^{'}$ satisfies $t_n\to\infty$, as $n\to \infty$. Then, $u(t_n,\cdot;u_0)\to \phi$ as $n\to\infty$. Since $u_n \to u_0$ as $n\to \infty$, for any fixed $k\in\mathbb{N}$, there is $u_{n_k}\in\{u_{n}\}$ such that $\|u(t_k,\cdot;u_{n_k})-u(t_k,\cdot;u_0)\|<\dfrac{1}{k}$. Therefore, $\|u(t_k,\cdot;u_{n_k})-\phi\|\to 0$ as $k\to \infty$. Without loss of generality, we may assume that  $u(t_n,\cdot;u_n )\to \phi$ as $n\to\infty$.} Note that  $t_n^{'}-t_n\to\infty$. (For otherwise, there exist subsequences $t_{n_k}^{'}$ and $t_{n_k}$ such that $t_{n_k}^{'}-t_{n_k}<c$ for some $c>0$. Then by choosing subsequences if necessary, $t_{n_k}^{'}-t_{n_k}\to c_1\in [0,c]$, it then leads to $u(t_{n_k}^{'},\cdot;u_{n_k} )\to \phi$, a contradiction).

Let $0<\delta\ll \|u_0-\phi\|$. For brevity, we assume that
$$
\|u_n -u_0\|<\delta, \quad \|u(t_n^{'},\cdot;u_n )-u_0\|<\delta,\quad \|u(t_n,\cdot;u_n )-\phi\|<\delta,\quad \forall\, n\ge 1.
$$
Then for any $n\geq 1$, there exists $\tau_n\in (t_n,t_n^{'})$ such that
$$
\|u(t,\cdot;u_n )-\phi\|<\delta\quad \forall\, t\in (t_n,\tau_n),\,\,
$$
and
$$
\|u(\tau_n,\cdot;u_n )-\phi\|=\delta.
$$
Similarly, $\tau_n-t_n\to \infty$ as $n\to\infty$.

Without loss of generality, assume that $u(\tau_n,\cdot;u_n )\to \tilde u$ ({since that \eqref{general-equa}+\eqref{autonomous} admits a global attractor, $\{u(\tau_n,\cdot;u_n )\}$ is precompact in $X$}). For any fixed $T_0>0$ and $n$ sufficiently large, one has $\tau_n-T_0>0$. By choosing subsequence still denotes it by $\tau_n$, one has $u(\tau_n-T_0,\cdot;u_n )\to u(-T_0,\cdot;\tilde u)$. Moreover, by the continuous dependence of solution of \eqref{general-equa}+\eqref{autonomous} on the initial value, $u(\tau_n+t,\cdot;u_n )\to u(t,\cdot;\tilde u)$ uniformly for $t\in [-T_0,T_0]$. Since $T_0>0$ is arbitrary, $u(t,\cdot;\tilde u)$ exists
for all $t<0$ and
$$\|u(t,\cdot;\tilde u)-\phi\|\le \delta\quad  \forall\,\, t\le 0.
$$

If $\phi$ is hyperbolic, then $u(t,\cdot;u_0)\in W^s_{loc}(\phi)$ for $t\gg 1$. By Lemma \ref{local-um}, one can choose $\delta,\epsilon>0$ sufficiently small such that $u(t,\cdot;\tilde u)\in W^u_{loc}(\phi)$ for $t\leq \epsilon$.
Assume that $0<t_0<\epsilon$ be such that
$u(t_0,\cdot;u_0)-\phi$ and $u(t_0,\cdot;\tilde u)-\phi$ have only simple zeros. Then
\begin{equation*}
\begin{split}
Z(u(t_0,\cdot;u_0)-\phi)&=Z(u(t_0,\cdot;u_n) -\phi)\ge Z(u(\tau_n+t_0,\cdot;u_n )-\phi)\\
&=Z(u(t_0,\cdot;\tilde u)-\phi)\ge Z(u(t_n^{'}+t_0,\cdot;u_n )-\phi)=Z(u(t_0,\cdot;u_0)-\phi)
\end{split}
\end{equation*}
for $n\gg 1$. Therefore, $Z(u(t_0,\cdot;u_0)-\phi)=Z(u(t_0,\cdot;\tilde u)-\phi)$. By Lemma \ref{zero-number} and zero number control on local invariant manifolds (see also \cite[Corrolary 3.5]{SWZ}), one has
$$
{ Z(u(t_0,\cdot;u_0)-\phi)\geq Z(u(t,\cdot;u_0)-\phi)>Z(u(t_0,\cdot;\tilde u)-\phi),\quad t\gg 1}
$$
which is a contradiction. Hence $\phi$ is not hyperbolic.

Let $a=\p_2 f(\phi,0)$, $b=\p_1 f(\phi,0) $ $w(t,x)=v(t,x-at)e^{-bt}$, then \eqref{fixed-point-linearized} can be transformed into $w_t=w_{xx}.$ Note that the eigenfunctions of $w_t=w_{xx}$ associated with the eigenvalue $\lambda_k=-k^2$, $k=0,1,\cdots$ are $w_k(t,x)=e^{-k^2t}\sin kx,e^{-k^2t}\cos kx$. Then it yields that the spectrum set $\sigma(a,b)=\{b,-1+b,\cdots,-k^2+b,\cdots\}$ and $E^k=\mathrm{span}\{\sin kx,\cos kx\}$ are the corresponding eigenspaces of \eqref{fixed-point-linearized}. Since $\dim E^{cu}>0$, if $\dim E^{u}=0$, then $b=0$ and $E^c=\mathrm{span}\{c\}$ with $c\neq 0$; if $\dim E^u>0$, then there exist some $k_0\in \mathbb{N}$ with $-k_0^2+b=0$ and $E^c=\mathrm{span}\{\sin k_0x,\cos k_0x\}$. Then, the constancy of zero number function $Z(\cdot)$ on $E^c$ can be obtained by using \cite[Theorem 2]{Angenent1988}.

The proof of lemma is completed.
\end{proof}

\begin{lem}\label{auton-inhomo-cons}
Assume that $u_0$ is neither a fixed point nor generates a rotating wave.
 \begin{itemize}
 \item[{\rm(i)}] If $\phi$ is spatially-inhomogeneous, then $Z(\sigma_a u(t,\cdot;u_0)-u(t,\cdot;\phi))=N$
for all { $t>0$}   and $a\in S^1$, where $N$ is as in Lemma \ref{fixed-constant} (i).

\item[{\rm(ii)}]
 If $\phi$ is spatially-homogeneous, then $Z(\sigma_a u(t,\cdot;u_0)-\phi)=N_0$
for all $a\in S^1$ and { $t>0$}, where  $N_0$ is as in Lemma \ref{fixed-constant} (ii).
\end{itemize}
\end{lem}

\begin{proof}
(i)  First, let $T=\frac{2\pi}{|c|}$ in the case that $u(t,x;\phi)=\phi(x-ct)$ for some $c\not =0$ and $T$ be any fixed positive number
in the case that $u(t,x;\phi)\equiv \phi(x)$.   Then $u(kT,\cdot;\phi)=\phi$ for all $k\in\NN$.
Let $u_n \in X$ and $t^{'}_n\to \infty$ be such that $u_n \to u_0$ and $u(t^{'}_n,\cdot;u_n )\to u_0$ as $n\to \infty$.
 Then
 $$
 \sigma_a u(t_n^{'},\cdot;u_n )\to \sigma_a u_0,\quad \forall\,\, a\in S^1.
 $$
 There are $\mathbb{N}\ni k_n\to\infty$ and $\tau_n\in [0,T)$ such that
 $$
 t_n^{'}=k_nT+\tau_n.
 $$
 Without loss of generality, we may assume that $\tau_n\to \tau\in [0,T]$. Then
 $$
 u(t_n^{'},\cdot;\phi)\to \phi(\cdot-c \tau).
 $$

Next, we prove that $Z(\sigma_a u(t,\cdot;u_0)-u(t,\cdot;\phi))$ is independent of $t>0$.

{ Choose any $t_0>0$ such that $\sigma_a u(t_0,\cdot;u_0)-u(t_0,\cdot;\phi)$ has only simple zeros.}
Suppose that there is {$\tau_0>t_0$} such that
$$
Z(\sigma_a u(t_0,\cdot;u_0)-u(t_0,\cdot;\phi))>Z(\sigma_a u(\tau_0,\cdot;u_0)-u(\tau_0,\cdot;\phi))
$$
with $\sigma_a u(\tau_0,\cdot;u_0)-u(\tau_0,\cdot;\phi)$ has only simple zeros.
Then
$$
{ Z(\sigma_a u(t_0,\cdot;u_0)-u(t_0,\cdot;\phi))=Z(\sigma_a u(t_0,\cdot;u_n) -u(t_0,\cdot;\phi))}>Z(\sigma_a u(\tau_0,\cdot;u_n )-u(\tau_0,\cdot;\phi))
$$
for $n\gg 1$. This implies that for any fixed $t\in\RR$,
$$
{ Z(\sigma_a u(t_0,\cdot;u_0)-u(t_0,\cdot;\phi))=Z(\sigma_a u(t_0,\cdot;u_n) -u(t_0,\cdot;\phi))}>Z(\sigma_a u(t+t_n^{'},\cdot;u_n )-u(t+t_n^{'},\cdot;\phi))
$$
for $n\gg 1$. Note that
$$
\sigma_a u(t+t_n^{'},\cdot;u_n )-u(t+t_n^{'},\cdot;\phi)\to \sigma_a u(t,\cdot;u_0)-\phi(\cdot-c(t+\tau)).
$$
Choose $t<0$ be such that $u(t,\cdot;u_0)-\phi(\cdot-c(t+\tau))$ has only simple zeros. Then
$$
Z(\sigma_a u(t+t_n^{'},\cdot;u_n )-u(t+t_n^{'},\cdot;\phi))=Z(\sigma_a u(t,\cdot;u_0)-\phi(\cdot-c(t+\tau))\ge Z(\sigma_au_0-\phi(\cdot-c\tau))
$$
for $n\gg 1$.
Hence
\begin{equation}
\label{zero-eq1}
{ Z(\sigma_au_0-\phi) \ge Z(\sigma_a u(t_0,\cdot;u_0)-u(t_0,\cdot;\phi))}  >Z(\sigma_a u_0-\phi(\cdot-c\tau)).
\end{equation}

Suppose that
$$
Z(\sigma_au(t_0,\cdot;u_0)-u(t_0,\cdot;\phi))=Z(\sigma_a u(t,\cdot;u_0)-u(t,\cdot;\phi))
$$
for all $t\ge t_0$. It  follows from the above arguments that
\begin{equation}
\label{zero-eq2}
{ Z(\sigma_au_0-\phi) \ge Z(\sigma_a u(t_0,\cdot;u_0)-u(t_0,\cdot;\phi))} \ge Z(\sigma_a u_0-\phi(\cdot-c\tau)).
\end{equation}

Case (ia). There are $p,q\in\NN$ such that $c\tau p=qT$. Assume that there is  { $\tau_0>t_0$} such that
$$
Z(\sigma_au(t_0,\cdot; u_0)-u(t_0,\cdot;\phi))>Z(\sigma_a u(\tau_0,\cdot;u_0)-u(\tau_0,\cdot;\phi)).
$$
By using \eqref{zero-eq1} and \eqref{zero-eq2} repeatedly,
$$
Z(\sigma_au_0-\phi)>Z(\sigma_a u_0-\phi(\cdot-c\tau))\ge Z(\sigma_a u_0-\phi(\cdot-2c\tau))\ge \cdots\ge Z(\sigma_a u_0-\phi(\cdot-c\tau p))
=Z(\sigma_au_0-\phi),
$$
which is a contradiction. Hence $\sigma_a u(t,\cdot;u_0)-u(t,\cdot;\phi)$ has only simple zeros for all $t\ge t_0$. This
implies that $\sigma_a u(t,\cdot;u_0)-u(t,\cdot;\phi)$ has only simple zeros for all $t>0$.

\medskip

Case (ib). $c\tau $ and $T$ are rationally independent. Assume that there is  { $\tau_0>t_0$} such that
$$
Z(\sigma_a u(t_0,\cdot;u_0)-u(t_0,\cdot;\phi))>Z(\sigma_a u(\tau_0,\cdot;u_0)-u(\tau_0,\cdot;\phi)).
$$
By using \eqref{zero-eq1} and \eqref{zero-eq2}  repeatedly again,
{ \begin{align*}
  Z(\sigma_au_0-\phi) &\ge Z(\sigma_a u(t_0,\cdot;u_0)-u(t_0,\cdot;\phi)) \\
&>Z(\sigma_a u_0-\phi(\cdot-c\tau))\ge Z(\sigma_a u_0-\phi(\cdot-2c\tau)\ge \cdots\ge Z(\sigma_a u_0-\phi(\cdot-c\tau k))\\
&\ge Z(\sigma_a u(t_0,\cdot;u_0)-u(t_0,\cdot;\phi(\cdot-c\tau k)).
\end{align*}}
Note that there is $k_n\to\infty$ such that
$$
\phi(\cdot-c\tau k_n)\to \phi(\cdot).
$$
Recall that { $\sigma_au(t_0,\cdot;u_0) -u(t_0,\cdot;\phi)$} has only simple zeros. Hence
$$
{  Z(\sigma_au(t_0,\cdot;u_0)-u(t,\cdot;\phi(\cdot-c\tau k_n)))=Z(\sigma_a u(t_0,\cdot;u_0)-u(t_0,\cdot;\phi))}
$$
for $n\gg 1$. It then follows that
$$
 Z(\sigma_au(t_0,\cdot;u_0)-u(t_0,\cdot;\phi)) >Z(\sigma_a u_0-\phi(\cdot-c\tau))\ge  Z(\sigma_au(t_0,\cdot;u_0)-u(t_0,\cdot;\phi)),
$$
which is a contradiction again. Hence, in this case,  $\sigma_a u(t,\cdot;u_0)-u(t,\cdot;\phi)$ has also simple zeros for $t\ge t_0$.
We have proved that
$\sigma_a u(t,\cdot;u_0)-u(t,\cdot;\phi)$ has only simple zeros for all $t>0$.

Now, we show that  $$
Z(\sigma_a u(t,\cdot;u_0)-u(t,\cdot;\phi))=N,\quad t>0,\ a\in S^1.
$$
{ First, by the above arguments, $\sigma_a u(t,\cdot;u_0)-u(t,\cdot;\phi)$ has only simple zeros for all $t>0$ and $a\in S^1$.}
Let $t_n\to\infty$ be such that
$$
  u(t_n,\cdot;u_0)\to  \phi,\quad \text{as } n\to\infty.
$$
Without loss of generality, assume that
$$
u(t_n,\cdot;\phi)\to \sigma_{\tilde a}\phi
$$
for some $\tilde a\in S^1$. Choose $a\not =\tilde a$.
Then
$$
\sigma_a u(t_n,\cdot;u_0)-u(t_n,\cdot;\phi)\to \sigma_a \phi-\sigma_{\tilde a}\phi
$$
as $n\to\infty$. By Lemma \ref{fixed-constant} (i),
$$
Z(\sigma_a \phi-\sigma_{\tilde a}\phi)=N.
$$
Therefore,
$$
Z(\sigma_a u(t,\cdot;u_0)-u(t,\cdot;\phi))=N
$$
for all $t>0$ and $a\in S^1$.
(i) is thus proved.

\medskip

(ii) By suitable modifications of the arguments in (i) before \eqref{zero-eq1},
$Z(\sigma_a u(t,\cdot;u_0)-\phi)$ is independent of $t>0$ and $a\in S^1$.
Let $ v(t,x)=u(t,x;u_0)-\phi$. Then $ v(t,x)$ satisfies
\begin{equation}\label{difference-equ}
v_t= v_{xx}+\hat a(t,x) v_x+\hat b(t,x) v,\quad x\in S^1,
\end{equation}
where
$$
\hat a(t,x)=\int_0^1{\p_2 f(\phi,s(u_x(t,x;u_0)))}ds
$$
and
$$
\hat b(t,x)=\int_0^1{\p_1 f(\phi+s(u(t,x;u_0)-\phi),u_x(t,x;u_0))}ds.
$$
{ Let $a=\p_2 f(\phi,0)$ and $b=\p_1 f(\phi,0)$.}
For simplicity, we assume that $b=k_0^2$, then by Lemma \ref{fixed-constant}, the Sacker-Sell spectrum of \eqref{fixed-point-linearized} can be written as $\sigma(\phi)=\{b,\cdots,-(k_0-1)^2+b,0,-(k_0+1)^2+b,\cdots\}$.
Choose  $0<-\lambda_0\ll 1$ be such that
\begin{equation}\label{fixed-point-linearized-trans}
v_t=v_{xx}+a v_x+(b+\lambda_0) v,\quad x\in S^1
\end{equation}
admits an exponential dichotomy and the Sacker-Sell spectrum $\{b+\lambda_{0},\cdots, -(k_0-1)^2+b+\lambda_0,\lambda_0, \cdots\}$ satisfying $-(k_0-1)^2+b+\lambda_0>0$.

By Lemma \ref{pertubation-invariant}, there is $\delta>0$ such that for any two bounded and uniformly continuous functions $a^0(t,x)\in C^1(\mathbb{R}\times S^1)$ and $b^0(t,x)\in C(\mathbb{R}\times S^1)$ satisfying $\|a^0(t,\cdot)-a \|<\delta$ and $\|b^0(t,\cdot)-b\|<\delta$, the following equation
\begin{equation}\label{difference-equ-2}
v_t= v_{xx}+\tilde a^0(t,x) v_x+(\tilde b^0(t,x)+\lambda_0)v,\quad x\in S^1,(\tilde a^0,\tilde b^0)\in H(a^0,b^0),
\end{equation}
admits an exponential dichotomy over $H(a^0,b^0)$.

Note that $\hat a(t,x)\to a$ and $\hat b(t,x)\to b$ as $t\to\infty$. Therefore, for any $\delta>0$, there exists $T>0$ such that $\|\hat a(t,x)-a\|<\delta$ and $\|\hat b(t,x)-b\|<\delta$, for $t>T$. Let $\varepsilon_a(t,x)=\hat a(t,x)-a$ and $\varepsilon_b(t,x)=\hat b(t,x)-b$.  Define the following two functions,
\begin{equation}\label{perturbation-coeff-a}
\tilde a(t,x)=\left\{
\begin{split}
 &a+\varepsilon_a(t,x),\quad t\geq T+1\\
 &a+(1+T-t)\varepsilon_a(T,x)+(t-T)\varepsilon_a(t,x),\quad T+1>t\geq T \\
 &a+\varepsilon_a(T,x),\quad t<T,
\end{split}\right.
\end{equation}
and
\begin{equation}\label{perturbation-coeff-b}
\tilde b(t,x)=\left\{
\begin{split}
 &b+\varepsilon_b(t,x),\quad t\geq { T+1}\\
 &b+(1+T-t)\varepsilon_b(T,x)+(t-T)\varepsilon_b(t,x),\quad T+1>t\geq T \\
 &b+\varepsilon_b(T,x),\quad t<T.
\end{split}\right.
\end{equation}
Since $\hat a(t,x)$ and $\hat b(t,x)$ are bounded and uniformly continuous functions and $\hat a(t,x)\in C(\mathbb{R}\times S^1)$, so are the functions $\tilde a(t,x)$ and $\tilde b(t,x)$. Moreover, $\|\tilde a(t,x)-a\|<\delta$ and $\|\tilde b(t,x)-b\|<\delta$ for $t\in\mathbb{R}$.
Therefore, the following equation
\begin{equation}\label{difference-equ}
v_t= v_{xx}+\tilde a^*(t,x) v_x+(\tilde b^*(t,x)+\lambda_0)v,\quad x\in S^1, (\tilde a^*,\tilde b^*)\in H(\tilde a,\tilde b)
\end{equation}
admits an exponential dichotomy over $H(\tilde a,\tilde b+\lambda_0)$ and $\dim E^u(\tilde a^*,\tilde b^*+\lambda_0)=\dim E^u(a,b)$. Moreover, it is obvious that $\dim E^u(\tilde a^*,\tilde b^*)\geq \dim E^u(\tilde a^*,\tilde b^*+\lambda_0)$. Therefore,  $\dim E^u(\tilde a^*,\tilde b^*)\geq \dim E^u(a,b)$ for all $(\tilde a^*,\tilde b^*)\in H(\tilde a, \tilde b)$.

Since $v(t,x)\to 0$, by the definition of center stable space over $H(\tilde a,\tilde b)$ (see also \eqref{infty-estimate}), it is not hard to see that $v(t,x)\in E^{cs}(\tilde a\cdot t,\tilde b\cdot t)$ for $t\geq T+1$. Thus, by zero number control on center stable spaces (see also Lemma \ref{zero-inva}), $Z(v(t,x))\geq N_0$ for $t\geq T+1$. This together with decreasing property of $Z(\cdot)$ implies that
$$
Z(u(t_0,\cdot;u)-\phi)\ge N_0,
$$
$t_0$ is sufficiently small as defined in the proof of Lemma \ref{fixed-constant}(ii).

Let $\tilde u$ be also as in the proof of Lemma \ref{fixed-constant}(ii).  Then
$$
Z(u(t_0,\cdot;u)-\phi)=Z(u(t_0,\cdot;\tilde u)-\phi).
$$
By similar arguments as in the above,
$$
Z(u(t_0,\cdot;\tilde u)-\phi)\le N_0.
$$
{Since $t_0>0$ can be arbitrary small}, it then follows that
$$
Z(\sigma_a u(t,\cdot;u_0)-\phi)=N_0
$$
for all $t>0$ and $a\in S^1$.
The lemma thus follows.
\end{proof}
\vskip 2mm

\begin{lem}\label{auton-T-const}
Assume that $u_0$ is not a fixed point and does not generate a rotating wave.
\begin{itemize}
\item[{\rm(i)}] Suppose $\phi$ is spatially-inhomogeneous. There are $T, T^*>0$ such that
$$
Z(u(t+T,\cdot;u_0)-\sigma_a u(t,\cdot;u_0))=N
$$
for all $t\geq T^*$ and $a\in S^1$, where $N$ is as in Lemma \ref{fixed-constant} (i).

\item[{\rm(ii)}] Suppose $\phi$ is spatially-homogeneous. There is $T>0$ such that
$$
Z(u(t+T,\cdot;u_0)-\sigma_a u(t,\cdot;u_0))=N_0
$$
for all $t\geq 1$ and $a\in S^1$, where $N_0$ is as in Lemma \ref{fixed-constant} (ii).
\end{itemize}
\end{lem}
\begin{proof}

(i)  First, note that for any fixed $t_0>0$ there is $T\gg 1$ (in the case that $u(t,x;\phi)=\phi(x-ct)$ with $c\not =0$,
$T=k\frac{2\pi}{|c|}$) such that
$$
\|u(T+t_0,\cdot;u_0)-\sigma_{\tilde a}\phi\|\ll 1
$$
for some $\tilde a\in S^1$. By { Lemma \ref{auton-inhomo-cons} (i)}
$$
Z(u(T+t_0,\cdot;u_0)-u(t_0,\cdot;u_0))=Z(\sigma_{\tilde a} \phi-u(t_0,\cdot;u_0))=N.
$$
Hence
$$
Z(u(t+T,\cdot;u_0)-u(t,\cdot;u_0))\le N\quad \forall \,\, t\ge t_0.
$$
Moreover, there is $T_0\geq t_0$ and some integer $N^*\leq N$ such that
\begin{equation}\label{constant-1}
Z(u(t+T,\cdot;u_0)-u(t,\cdot;u_0))=N^*,\quad \forall t\geq T_0.
\end{equation}
Recall that there is $t_n\to\infty$ such that
$$
u(t_n,\cdot;u_0)\to \phi.
$$
Hence for any $ a\in S^1\setminus\{0\}$,
$$
u(t_n+T,\cdot;u_0)-\sigma_{  a} u(t_n,\cdot;u_0)\to  \phi-\sigma_{a}\phi.
$$
This implies that
\begin{equation}\label{tranlation-const}
 Z(u(t+T,\cdot;u_0)-\sigma_{ a} u(t,\cdot;u_0))=N
\quad \forall\,\, t\gg 1.
\end{equation}
By \eqref{constant-1}, there exists $\delta_0>0$ such that for any $0<a<\delta_0$, one has
\[
Z(u(T_0+T,\cdot;u_0)- \sigma_au(T_0,\cdot;u_0))=N^{*},
\]
which combining with \eqref{tranlation-const} lead to that $N^{*}=N$. Therefore,
$$
Z(u(t+T,\cdot;u_0)- u(t,\cdot;u_0))=N\quad \forall\,\, t\geq t_0.
$$
Note that $S^1$ is compact, it then follows from Heine-Borel theorem that there is $T^*>0$ such that
$$
Z(u(t+T,\cdot;u_0)-\sigma_a u(t,\cdot;u_0))=N
$$
for $t\ge T^*$ and $a\in S^1$.

\medskip
(ii) By { Lemma \ref{auton-inhomo-cons} (ii)},
$$
Z(\sigma_au(t,\cdot;u_0)-\phi)=N_0,\quad \forall\ { t>0}, \ a\in S^1
$$
Choose $T\gg 1$ be such that
$$
\|\sigma_a u(1+T,\cdot;u_0)-\phi\|\ll 1.
$$
Then
$$
Z(u(1,\cdot;u_0)-\sigma_a u(1+T,\cdot;u_0))=N_0.
$$
This implies that
$$
Z(u(t,\cdot;u_0)-\sigma_a u(t+T,\cdot;u_0))\le N_0,\quad \forall\ t\geq 1, \ a\in S^1.
$$

For given $a\in S^1$, let
$v(t,x)=\sigma_a u(t+T,x;u_0)-u(t,x;u_0)$. Then $v(t,x)$ satisfies
\[
v_t=v_{xx}+a(t,x)v_x+b(t,x)v,\quad x\in S^1,
\]
where {
$$
a(t,x)=\int_0^1{\p_2 f(u(t,x;u_0),s(\sigma_a u_x(t+T,x;u_0)-u_x(t,x;u_0)))}ds
$$
and
$$
b(t,x)\\=\int_0^1{\p_1 f(u(t,x;u_0)+s(\sigma_a u(t+T,x;u_0)-u(t,x;u_0)),\sigma_au_x(t+T,x;u_0))}ds.
$$}

Note that $a(t,x)\to \p_2 f(\phi,0)$ and $b(t,x)\to \p_1 f(\phi,0)$ as $t\to\infty$, and
$v(t,x)\to 0$ as $t\to\infty$. Reasoning as in the proof of Lemma \ref{auton-inhomo-cons} (ii), one has
$$
Z(\sigma_a u(t+T,\cdot;u_0)-u(t,\cdot;u_0))\ge N_0.
$$
(ii) then follows.
\end{proof}
\vskip 2mm

We now prove Theorem \ref{aut-thm}.

\begin{proof}[Proof of Theorem \ref{aut-thm}]
Assume that $u_0$ is neither a fixed point nor generates a rotating wave.
By the above lemmas, there are $T,T^*\geq 0$, and $N^*\in\NN$ such that
\begin{equation}\label{constant-sufficient-L}
Z(\sigma_a u(t+T,\cdot;u_0)-u(t,\cdot;u_0))=N^*,\quad \forall\ t\ge T^*,\ a\in S^1.
\end{equation}
Therefore, $\max u(t+T,\cdot;u_0)\not =\max u(t,\cdot;u_0)$ for $t\ge T^*$. Without loss of generality, we may assume that $T^*=0$ and $\max u(t+T,\cdot;u_0)>\max u(t,\cdot;u_0)$ for $t\geq 0$. Then
\begin{equation}
\label{aux-new-eq1}
 \max u(t,\cdot;u_0)<\max u(t+T,\cdot;u_0)<\max u(t+2T,\cdot;u_0)<\cdots<\max u(t+kT,\cdot;u_0)<\cdots
 \end{equation}
 for all $t\ge 0$.

Let $u_n $ and $t_n^{'}\to \infty$ be such that
$$
u_n \to u_0,\quad u(t_n^{'},\cdot;u_n )\to u_0.
$$
Then
$$
\sigma_a u(T,\cdot;u_n )-u_n \to \sigma_a u(T,\cdot;u_0)-u_0
$$
and
$$
\sigma_a u(t_n^{'}+T,\cdot;u_n )-u(t_n^{'},\cdot;u_n) \to \sigma_a u(T,\cdot;u_0)-u_0.
$$
Hence there is $K_0>0$ such that
\begin{equation}\label{const-seuqence}
Z(\sigma_a u(t+T,\cdot;u_n)-u(t,\cdot;u_n ))=N^*
\end{equation}
for $a\in S^1$, $0\le t\le t_n^{'}$, and $n\ge K_0$. This implies
that for any fixed $\tau\in [0,T)$, $\max u(kT+\tau,\cdot;u_n )$ is increasing in $k\in\NN$ provided that $kT+\tau\le t_n^{'}$ and $n\gg 1$. Let $t_n^{'}=k_nT+\tau_n$, where $k_n\in\NN$, $\tau_n\in [0,T)$. Without loss of generality, we may assume that $\tau_n\to \tau_0\in [0,T]$ (i.e., $\tau_0-\tau_n\to 0$). Then $u(k_nT+\tau_0,\cdot;u_n )\to u_0$. Noticing that
$\max u_n  \to \max u_0$, one has
$$
\max u(k_nT+\tau_0,\cdot;u_n )\to \max u_0.
$$
Moreover, we assume that $\tau_0\neq 0$ (Since if $\tau_0=0$, it is easy to combine with \eqref{aux-new-eq1} to get a contradiction). It then follows that
\begin{equation}\label{max-inequal}
\begin{split}
\max u(\tau_0,\cdot; u_0)&\leftarrow \max u(\tau_0,\cdot; u_n )\\
&<\max u(T+\tau_0,\cdot;u_n )\\
&<\cdots\\
&<\max u(k_nT+\tau_0,\cdot;u_n )\\
&\to \max u_0.
\end{split}
\end{equation}
As a matter of fact, by \eqref{constant-sufficient-L}, for any $k\in \mathbb{N}$, there is $K(k)>0$ such that
$$
Z(\sigma_a u(t+T,\cdot;u(k\tau_0,\cdot;u_n))-u(t,\cdot;u(k\tau_0,\cdot;u_n)))=N^*
$$
for $a\in S^1$, $0\le t\le t_n^{'}$, and $n\ge K(k)$. Repeating the arguments between equations \eqref{const-seuqence} and \eqref{max-inequal}, one can further get
$$
\max u(k\tau_0,\cdot; u_0)<\max u((k-1)\tau_0,\cdot; u_0).
$$
Therefore,
\begin{equation}
\label{aux-new-eq2}
\max u_0>\max u(\tau_0,\cdot;u_0)>u(2\tau_0,\cdot;u_0)>\cdots> u(k\tau_0,\cdot;u_0)>\cdots.
\end{equation}

Assume that $\tau_0$ and $T$ are rationally dependent. Then there are $p,q\in\NN$ such that $p\tau_0=q T$. This together with \eqref{aux-new-eq2}
implies that
$$
\max u_0>\max u(qT,\cdot;u_0),
$$
which contradicts to \eqref{aux-new-eq1}.

Assume that $\tau_0$ and $T$ are rationally independent. Then there are $p_n,q_n\in \NN\setminus \{1\}$ and
$r_n\to 0$ such that $p_n \tau_0=q_n T+r_n$. This together with \eqref{aux-new-eq1} and \eqref{aux-new-eq2} implies that
$$
\max u_0>\max u(\tau_0,\cdot;u_0)>\max u(p_n\tau_0,\cdot;u_0)=\max u(q_nT+r_n,\cdot;u_0)>\max u(r_n,\cdot;u_0) \to \max u_0,
$$
which is a contradiction. Therefore, we must have $u_0$ is a fixed point or generating a rotating wave.
\end{proof}

\section{The time periodic case}
In this section, we consider the non-wandering points of \eqref{general-equa}+\eqref{periodic-eq}. Since some of the deductions are very similar to {those} in the autonomous case, we will focus on the places where the deductions are different, and only outline the idea of their derivations when the proofs are analogous to {those} in the autonomous case.

 The associated Poincar\'{e} map $P$ of \eqref{general-equa}+\eqref{periodic-eq} is defined as follows:
$$
Pu_0=u(T,\cdot;u_0),\quad u_0\in X.
$$
A compact invariant set $A$ of $P$ is said to be a {\it maximal compact invariant set} if every compact invariant set of $P$ is a subset of $A$. An invariant set $\mathcal{A}$ of $P$  is said to be a {\it global attractor} {(see, e.g. \cite[p.17]{Hale88})} of \eqref{general-equa}+\eqref{periodic-eq}  if $\mathcal{A}$   is a maximal compact invariant set which attracts each bounded set $B\subset X$ under $P$, that is
$$
\lim_{n\to\infty}\sup_{u'\in B}\inf_{u\in\mathcal{A}}\|P^n(u')-u\|=0.
$$

 Throughout this section, we always assume that  \eqref{general-equa}+\eqref{periodic-eq} admits a compact global attractor $\mathcal{A}$.
 Then, the following family of equations:
\begin{equation}\label{general-induced-equa}
u_{t}=u_{xx}+f(t+\tau,u,u_{x}),\,\,t>0,\, \tau\in[0,T),\,x\in S^{1}=\mathbb{R}/2\pi \mathbb{Z},
\end{equation}
also admit global attractors $\mathcal{A}_\tau=u(\tau,\cdot;\mathcal{A})$.

\begin{defn}\label{non-wan-peri}
{\rm
A point $u_0\in X$ is said to be a non-wandering point of \eqref{general-equa}+\eqref{periodic-eq} {if for any neighborhood
 $U$ of $u_0$, and any integer $N_0>0$, there exists an integer $n>N_0$ such that $\{u(nT,\cdot;u')\,|\, u'\in U\}\cap U\not =\emptyset$ (see, e.g. \cite[p.106]{wiggins}).}}
  \end{defn}

Obverse that if $u_0\in X$ is a non-wandering point of \eqref{general-equa}+\eqref{periodic-eq}, then there are $X\ni u_n\to u_0$ and $k_n\in \mathbb{N}$ with $k_n\to \infty$ such that $u(k_nT,\cdot;u_n)\to u_0$ as $n\to\infty$.  Moreover, similar to the arguments after Definition \ref{non-wan-auto}, $u(t,\cdot;u_0)$ also exists for all $t\in \mathbb{R}$.

In the rest of this section, we assume that  $u_0\in X$  is a non-wandering point of
\eqref{general-equa}+\eqref{periodic-eq}. Let $\tilde \omega(u_0)$ be the $\omega$-limit set of $u_0$ with respect to the time $T$-map or Poincar\'{e} map $P$. Then by \cite[Theorem 1]{SF1}
\begin{equation}
\label{omega-limit-set-eq2}
\tilde \omega(u_0)\subset \Sigma \phi=\{\sigma_a \phi\,|\, a\in S^1\},
\end{equation}
where $\phi\in X$. Moreover, there is {$r\in S^1$} such that $P w=\sigma_r w$ for any $w\in \tilde \omega(u_0)$.

In the following, $\phi$ is assumed to be as in \eqref{omega-limit-set-eq2}. We say that $u_0$ is  a {{\it $T$-periodic point}} if $P u_0=u_0$
and that $u_0$ {\it generates a rotating wave on the torus}  if $u_0(\cdot)\not\equiv {\rm const}$ and $P u_0=\sigma_ r u_0$ for some $r\in S^1$ with $\sigma_r u_0\not =u_0$.
\vskip 2mm

We have the following main result
\medskip
\begin{thm}\label{peri-thm} Suppose that  $u_0\in X$  is a non-wandering point of
\eqref{general-equa}+\eqref{periodic-eq}. Then
$u_0$ is a {$T$-periodic point} or generates a rotating wave on the torus.
\end{thm}

Before proving the above theorem, we prove some lemmas.

\begin{lem}\label{period-center-constant}
\begin{itemize}
\item[{\rm(i)}] If $\phi$ is spatially-inhomogeneous, then there is $N\in\NN$ such that
$$
Z(\phi-\sigma_a \phi)=N
$$
for any $a\in S^1\setminus\{0\}$.
\item[{\rm(ii)}] If $\phi$ is spatially-homogeneous, {then $\phi$ is neither stable nor hyperbolic.  And either $\dim E^u=0$ with $\dim E^c=1$ or $\dim E^u$ is odd with $\dim E^s=2$,} and there is $N_0\in \NN$ such that
$$
Z(v(\cdot))=N_0
$$
for any $v\in E^c\setminus\{0\}$.
\end{itemize}
\end{lem}

\begin{proof}
The proof of this lemma is similar to that in Lemma \ref{fixed-constant}.

(i) Noticing that $u(nT,x;\sigma_a\phi)=\sigma_{a+nr}\phi(x)$ for all $n\in\ZZ$,
$$
u(nT,x;\phi)-u(nT,x;\sigma_a\phi)=\sigma_{nr}\phi(x)-\sigma_{a+nr}\phi(x).
$$
Without loss of generality, we assume that $\lim_{n\to\infty}nr=0 (\mathrm{mod} 2\pi)$ no matter whether $r$ and $\pi$ is rationally dependent. Therefore,
$$
u(nT,\cdot;\phi)-u(nT,\cdot;\sigma_a\phi)=\sigma_{nr}\phi(\cdot)-\sigma_{a+nr}\phi(\cdot)\to \phi-\sigma_a\phi
$$
as $n\to \infty$.

{ By Lemma \ref{zero-number},  for any given $a\in S^1\setminus\{0\}$, there is $t_a>0$ such that both $u(-t_a,\cdot;\phi)-u(-t_a,\cdot;\sigma_a\phi)$ and $u(t_a,\cdot;\phi)-u(t_a,\cdot;\sigma_a\phi)$ have only simple zeros on $S^1$. Observe that there also exist $T_a>0$ and $N_a\in\mathbb{N}$ such that
\begin{equation*}
Z(u(t,\cdot;\phi)-u(t,\cdot;\sigma_a\phi))=N_a, \, t\geq T_a.
\end{equation*}
Therefore,
\begin{equation*}
\begin{split}
  N_a &=Z(u(nT-t_a,\cdot,\phi)-u(nT-t_a,\cdot,\sigma_a\phi))=Z(u(-t_a,\cdot,\phi)-u(-t_a,\cdot,\sigma_a\phi))\\
  &\geq Z(\phi-\sigma_a\phi)\geq Z(u(t_a,\cdot,\phi)-u(t_a,\cdot,\sigma_a\phi))\\
  &=Z(u(nT+t_a,\cdot,\phi)-u(nT+t_a,\cdot,\sigma_a\phi))=N_a, \quad n\gg 1.
\end{split}
\end{equation*}
}
Hence $\phi-\sigma_a \phi$ has only simple zeros for any $a\in S^1\setminus\{0\}$. (i) then follows.

\medskip

(ii) First, it is easy to obtain $\dim E^{cu}>0$ by using similar deduction as in Lemma \ref{fixed-constant}(ii).

Next we prove that $\phi$ cannot be hyperbolic. Note that $u(t+T,\cdot;\phi)=u(t,\cdot;\phi)$ and $u(nT,\cdot;u_0)\to \phi$ as $n\to\infty$. Let $k_n >n$ be such that
$$
u(k_nT,\cdot;u_n )\to u_0
$$
as $n\to\infty$. Without loss of generality, we may assume that $u(nT,\cdot;u_n )\to \phi$ as $n\to\infty$. Then, reasoning as the proof in Lemma \ref{fixed-constant}(ii), $k_n-n\to\infty$.

Let $0<\delta\ll \|u_0-\phi\|$. For convenience, we assume that
$$
\|u_n -u_0\|<\delta, \quad \|u(k_n T,\cdot;u_n )-u_0\|<\delta,\quad \|u(nT,\cdot;u_n )-\phi\|<\delta\quad \forall\, n\ge 1.
$$
Then there are $\tilde k_n\in [n,k_n)$ such that
$$
\|u(k_n^{'}T,\cdot;u_n )-\phi\|<\delta\quad \forall\, k_n^{'}\in [n,\tilde k_n],\,\, n\ge 1,
$$
and
$$
\|u((\tilde k_n+1)T,\cdot;u_n )-\phi\|\ge \delta\quad \forall\, n\ge 1.
$$
Moreover, $\tilde k_n-n\to \infty$ as $n\to\infty$.

Since \eqref{general-equa}+\eqref{periodic-eq} admits a global attractor, for any $t\in \mathbb{R}$ and $n\gg 1$, $u(\tilde k_n T+t,\cdot;u_n )$ is well defined and precompact in $X$, for simplicity, assume that $u(\tilde k_n T+t,\cdot;u_n )\to u(t,\cdot,\tilde u_0)$. Thus, $u(t,\cdot;\tilde u_0)$ exists
for all $t\in\mathbb{R}$ and
$$\|u(kT,\cdot;\tilde u_0)-\phi\|\le \delta\quad  \forall\,\, k\in\ZZ^-.
$$

If $\phi$ is hyperbolic, then $u(nT,\cdot;u_0)\in W_{loc}^s(\phi)$ for $n\gg 1$. By choosing $\delta>0$ sufficiently small, $\tilde u_0\in W^u_{loc}(\phi)$. Also we can choose $0<t_0\ll 1$ be such that
$u(t_0,\cdot;u_0)-u(t_0,\cdot;\phi)$ and $u(t_0,\cdot;\tilde u)-u(t_0,\cdot;\phi)$ have only simple zeros. Since $\phi$ is spatially-homogeneous
\begin{equation*}
\begin{split}
Z(u(t_0,\cdot;u_0)-u(t_0,\cdot;\phi))&=Z(u(t_0,\cdot;u_n)-u(t_0,\cdot;\phi))\ge Z(u(\tilde k_nT+t_0,\cdot;u_n )-u(t_0,\cdot;\phi))\\
&=Z(u(t_0,\cdot;\tilde u)-u(t_0,\cdot;\phi))\ge Z(u(k_nT+t_0,\cdot;u_n )-u(t_0,\cdot;\phi))\\
&=Z(u(t_0,\cdot;u_0)-u(t_0,\cdot;\phi))
\end{split}
\end{equation*}
for $n\gg 1$. Therefore, $Z(u(t_0,\cdot;u_0)-\phi)=Z(u(t_0,\cdot;\tilde u)-\phi)$. On the other hand, by zero number control on local invariant manifolds (see also \cite[Corrolary 3.5]{SWZ}),
$$
{ Z(u(t_0,\cdot;u_0)-u(t_0,\cdot;\phi))\geq Z(u(nT,\cdot;u_0)-\phi)> Z(\tilde u_0-\phi)\geq Z(u(t_0,\cdot;\tilde u)-u(t_0,\cdot;\phi)),\quad n\gg 1},
$$
which is a contradiction. Hence $\phi$ is not hyperbolic.

The rest of the proof is almost the same as those in Lemma \ref{fixed-constant} (ii), we omit it here.
\end{proof}

\begin{lem}
\begin{itemize}\label{peri-const}
\item[{\rm(i)}]  If $\phi$ is spatially-inhomogeneous, then
$$Z(\sigma_a u(t,\cdot;u_0)-u(t,\cdot;\phi))=N$$
for all { $t>0$} and $a\in S^1$, where $N$ is as in Lemma \ref{period-center-constant}(i).

\item[{\rm(ii)}] If $\phi$ is spatially-homogeneous, then
$$
Z(\sigma_a u(t,\cdot;u_0)-u(t,\cdot;\phi))=N_0
$$
for all $a\in S^1$ and $t>0$, where $N_0$ is as in Lemma \ref{period-center-constant}(ii).
\end{itemize}
\end{lem}

\begin{proof}
It can be proved by similar arguments as those in Lemma \ref{auton-inhomo-cons}.

(i) Fix any $a\in S^1$ and any $t_0>0$ such that $\sigma_a u(t_0,\cdot;u_0)-u(t_0,\cdot;\phi)$ has only simple zeros. Let $k_nr=2k^{'}_n\pi+\tau_n$, where $\tau_n\in [0,2\pi)$. Without loss of generality assume that $\tau_n\to \tau$, then $u(k_nT,\cdot;\phi )\to \phi(\cdot+\tau)$. Moreover, by replacing $t'_n$ with $k_nT$ in the proving of Lemma \ref{auton-inhomo-cons}(i), we also have the following conclusions:

If there is $\tau_0>t_0$ such that
$$
Z(\sigma_au(t_0,\cdot; u_0)-u(t_0,\cdot;\phi))>Z(\sigma_a u(\tau_0,\cdot;u_0)-u(\tau_0,\cdot;\phi))
$$
with $\sigma_a u(\tau_0,\cdot;u_0)-u(\tau_0,\cdot;\phi)$ has only simple zeros. Then
\begin{equation}
\label{zero-eq3}
Z(\sigma_au_0-\phi) \ge Z(\sigma_a u(t_0,\cdot;u_0)-u(t_0,\cdot;\phi)) >Z(\sigma_a u_0-\phi(\cdot+\tau)).
\end{equation}
If
$$
Z(\sigma_a u(t_0,\cdot;u_0)-u(t_0,\cdot;\phi))=Z(\sigma_a u(t,\cdot;u_0)-u(t,\cdot;\phi)),\quad \forall t\ge t_0.
$$
Then
\begin{equation}
\label{zero-eq4}
Z(\sigma_a u_0-\phi)\ge Z(\sigma_a u_0-\phi(\cdot+\tau)).
\end{equation}
The remaining  of the proof for (i) can be divided into two cases: (ia) there are $p,q\in\NN$ such that $\tau=\frac{p\pi}{q}$; (ib) $\frac{\tau}{\pi}$ is an irrational number. While the proof for these two cases are also very similar to the deductions in Lemma \ref{auton-inhomo-cons}(i), we also omit it here.

\medskip

(ii) Let $ v(t,x)=u(t,x;u_0)-u(t,x;\phi)$. Then $ v(t,x)$ satisfies
\begin{equation}
v_t= v_{xx}+a(t,x) v_x+b(t,x) v,\quad x\in S^1
\end{equation}
where
$$
a(t,x)=\int_0^1{\p_3 f(t,u(t,x;u_0),su_x(t,x;u_0))}ds
$$
and
$$
b(t,x)\\=\int_0^1{\p_2 f(t,u(t,x;u_0)+s(u(t,x;u_0)-u(t,x;\phi)),u_x(t,x;u_0))}ds.
$$

Observing that { $a(t,x)-\p_3 f(t,u(t,x;\phi),0)\to 0$}, { $b(t,x)- \p_2 f(t,u(t,x;\phi),0)\to 0$}, and $v(t,x)\to 0$ as $t\to\infty$. Similarly as the arguments in Lemma \ref{auton-inhomo-cons}(ii),
$$
Z( u(t_0,\cdot;u_0)-u(t_0,\cdot;\phi))\ge N_0,
$$
{where $t_0$ is as in (i).}
Then, by the arguments of Lemma \ref{period-center-constant}(ii),
$$
Z(u(t_0,\cdot;\tilde u_0)-u(t_0,\cdot;\phi))=Z(u(t_0,\cdot;u_0)-u(t_0,\cdot;\phi)).
$$
Let $t\to-\infty$, by similar arguments as in the above,
$$
Z( u(t_0,\cdot;\tilde u_0)-u(t_0,\cdot;\phi))\le N_0.
$$
Therefore,
$$
Z(u(t_0,\cdot;\tilde u_0)-u(t_0,\cdot;\phi))=N_0.
$$
Since $t_0>0$ can be arbitrary small and $\phi$ is spatially-homogeneous. It then follows that
$$
Z(\sigma_a u(t,\cdot;u_0)-u(t,\cdot;\phi))=N_0
$$
for all $t>0$ and $a\in S^1$.
\end{proof}

\begin{lem}\label{peri-mT-const}
\begin{itemize}
\item[{\rm (i)}] Assume that $\phi$ is spatially-inhomogeneous.  There are $m\ge 1$ and $T^*>0$ such that
$$
Z(u(t+mT,\cdot;u_0)-\sigma_a u(t,\cdot;u_0))=N
$$
for all $t\geq T^*$ and $a\in S^1$, where $N$ is as in Lemma \ref{period-center-constant}(i).
\item[{\rm (ii)}] Assume that $\phi$ is spatially-homogeneous.  There is $m\ge 1$ such that
$$
Z(u(t+mT,\cdot;u_0)-\sigma_a u(t,\cdot;u_0))=N_0
$$
for all $t>0$ and $a\in S^1$, where $N_0$ is as in Lemma \ref{period-center-constant}(ii).
\end{itemize}
\end{lem}
\medskip

\begin{proof}
(i) By Lemma \ref{peri-const}(i), { for any fixed $t_0>0$,
 $\sigma_a u(t_0,\cdot;u_0)-\sigma_bu(t_0,\cdot; \phi)$} has only simple zeros and $Z(\sigma_a u(t_0,\cdot;u_0)-u(t_0,\cdot;\sigma_b \phi))=N$ for any $a,b\in S^1$.
 Note that
 $$
 d(u(mT,\cdot;u_0),\tilde \omega(u_0))\to 0
 $$
 as $m\to\infty$. Therefore,
\begin{equation}\label{peri-ays-const}
 { Z(\sigma_a u(t_0,\cdot;u_0)- u(t_0+mT,\cdot;u_0))=Z(\sigma_a u(t_0,\cdot;u_0)-u(t_0,\cdot;\phi))}=N\quad \forall \, m\gg 1,\,  a\in S^1.
\end{equation}
It then follows from Lemma \ref{zero-number} that
\begin{equation}
\label{new-eq1}
Z(u(t+mT,\cdot;u_0)-\sigma_a u(t,\cdot;u_0))\le N,\quad \forall \,\, t\ge t_0,\, a\in S^1,
\end{equation}
for $m\gg 1$.

Now, fix $m$ in \eqref{peri-ays-const} and let $\tilde a$ be such that
$$
u(mT,\cdot;\phi)=\sigma_{\tilde a}\phi.
$$
Note also that
$$
u(k_nT,\cdot;u_0)\to \phi
$$
for some $k_n\to\infty$. Hence
$$
u(k_nT+mT,\cdot;u_0)-\sigma_{  a} u(k_nT,\cdot;u_0)\to  \sigma_{\tilde a}\phi -\sigma_{ a}\phi.
$$
This implies that
\begin{equation}
\label{new-eq2}
Z(u(t+mT,\cdot;u_0)-\sigma_{ a} u(t,\cdot;u_0))=N,
\quad \forall\,\, t\gg 1,\, a\not =\tilde a.
\end{equation}
Since there are $T_{\tilde a}>0$ and $N_{\tilde a}\in\mathbb{N}$ such that
\[
Z(u(t+mT,\cdot;u_0)-\sigma_{ \tilde a} u(t,\cdot;u_0))=N_{\tilde a},\quad \forall t\geq T_{\tilde a}.
\]
By the continuity of $Z(\cdot)$, there is $\delta_{\tilde a}>0$ such that
\[
Z(u(T_{\tilde a}+mT,\cdot;u_0)-\sigma_{ a} u(T_{\tilde a},\cdot;u_0))=N_{\tilde a}, \quad  |a-\tilde a|<\delta_{\tilde a}.
\]
This combining with \eqref{new-eq1} and \eqref{new-eq2} lead to $N_{\tilde a}=N$; moreover, one has
\begin{equation}
\label{new-eq3}
Z(u(t+mT,\cdot;u_0)-\sigma_{ a} u(t,\cdot;u_0))=N
\quad \forall\,\, t\gg 1,\, a\in S^1.
\end{equation}
Therefore, there is $T^*>0$ such that
$$
Z(u(t+mT,\cdot;u_0)-\sigma_a u(t,\cdot;u_0))=N
$$
for $t\ge T^*$ and $a\in S^1$.

\medskip
(ii) Fixed some $t_0>0$,  by Lemma \ref{peri-const}(ii)
$$
Z(u(t,\cdot;u_0)-u(t,\cdot;\phi))=N_0
$$
for all { $t\geq t_0$}. Choose $m\gg 1$ be such that
$$
\|\sigma_a u(mT,\cdot;u_0)-\phi\|\ll 1,\quad \forall a\in S^1.
$$
Then
$$
{ Z(u(t_0,\cdot;u_0)-\sigma_a u(t_0+mT,\cdot;u_0))=N_0.}
$$
Furthermore, by Lemma \ref{zero-number},
$$
Z(u(t,\cdot;u_0)-\sigma_a u(t+mT,\cdot;u_0))\le N_0
$$
for $t\ge t_0$ and $a\in S^1$.

Given $a\in S^1$, let
$v(t,x)=\sigma_a u(t+mT,x;u_0)-u(t,x;u_0)$. Then $v(t,x)$ satisfies
$$
v_t=v_{xx}+a(t,x)v_x+b(t,x)v,\quad x\in S^1,
$$
where
$$
a(t,x)=\int_0^1{\p_3 f(t,u(t,x;u_0),s(\sigma_a u_x(t+mT,x;u_0)-u_x(t,x;u_0)))}ds
$$
and
$$
b(t,x)\\=\int_0^1{\p_2 f(t,u(t,x;u_0)+s(\sigma_a u(t+mT,x;u_0)-u(t,x;u_0)),\sigma_au_x(t+mT,x;u_0))}ds.
$$
Note that $a(t,x)- \p_3 f(u(t,\cdot;\phi),0)\to 0$, $b(t,x)- \p_2 f(t,u(t,\cdot;\phi),0)\to 0$ and
$v(t,x)\to 0$ as $t\to\infty$. By zero number control on invariant spaces (see Lemma \ref{zero-inva}),
$$
Z(\sigma_a u(t+mT,\cdot;u_0)-u(t,\cdot;u_0))\ge N_0, \quad t\gg 1.
$$
(ii) then follows.
\end{proof}

\begin{proof} [Proof of Theorem \ref{peri-thm}]
 Let $u_0$ be a non-wandering point of the Poincar\'{e} map $P$. Then replace $T$, $t'_n$ by $mT$ and $k'_n$ in the proof of Theorem \ref{aut-thm} separately and let $\tau_n=m_0T\in \{0,T,2T,\cdots,(m-1)T\}$. {Theorem \ref{peri-thm} is then immediately obtained by using similar arguments as in Theorem \ref{aut-thm}}.
\end{proof}

\section{Concluding Remarks}

In this section, we give several concluding remarks on the assumptions,  the main results, and
  the techniques established in this paper.

  \medskip

\noindent {\bf Remark 1.} This remark is about the existence of compact global attractors.
In order to ensure that the equation \eqref{general-equa} admits a compact global attractor, a common admissible condition (see \cite{CCJ,Chen-P,JR2,Pol} for similar assumptions) is to assume that there exist constants $\epsilon>0$, $\delta>0$ and a continuous function $C(\cdot)$ maps $[0,\infty)$ to $[0,\infty)$, such that $f$ satisfies the following:
\begin{equation*}\label{dissipative-condition}
\begin{split}
&\forall \ l>0, \ p \in \mathbb{R}, \quad \sup_{(t,u)\in\mathbb{R}\times[-l,l]} \quad|f(t, u, p)| \leq C(l)\left(1+|p|^{2-\varepsilon}\right),  \\
& \forall\ |u| \geq \delta, \ t \in \mathbb{R}, \quad u f(t, u, 0) \leq 0.
\end{split}
\end{equation*}
Thus, \eqref{general-equa}+\eqref{autonomous} or \eqref{general-equa}+\eqref{periodic-eq} will admit a compact global attractor (see
\cite{CCJ,Pol}).
\vskip 2mm

\noindent {\bf Remark 2.} This remark is about the non-wandering set of \eqref{general-equa}+\eqref{autonomous}.
Let $f$ be such that \eqref{general-equa}+\eqref{autonomous} admits a compact global attractor. Then, Theorem \ref{aut-thm} tells us that the non-wandering set only consists of limit points. While in \cite{JR2}, to obtain a characterization of the structure of non-wandering set, the function $f=f(x,u,u_x)$ in \eqref{general-equa}+\eqref{autonomous} was further assumed to be specially ``{\it generic}" so that any equilibrium point or periodic point is hyperbolic, that there is no homoclinic orbit, and that all the heteroclinic orbits are transversal; and then, the non-wandering set \eqref{general-equa}+\eqref{autonomous} was proved to consist of finite many hyperbolic fixed points. Therefore, the characterization of non-wandering set we obtained here is more general than that in \cite{JR2}, when $f=f(u,u_x)$.

Since the approach of our research are different from that in \cite{JR2}, asymptotic behavior of the orbits nearby hyperbolic points as well as non-hyperbolic points {are} both considered here. For general $f=f(u,u_x)$, due to the complexity of the orbits near the non-hyperbolic points, the characterization {turns} to be much more difficult. In fact, we have already seen that the asymptotic behaviors near hyperbolic points are relatively clear, while the main difficulty is studying the asymptotic behaviors of orbits near non-hyperbolic points. For example, in Lemma \ref{fixed-constant}(ii) (see also Lemma \ref{period-center-constant}(ii) for time periodic case), if a non-wandering point approachs a hyperbolic fixed point, then itself must be a limit point. Thus, we only {considered} the non-wandering point which {would} not converge to a hyperbolic fixed point.
Some new methods were introduced in our study, for instance, we {used} the roughness of exponential dichotomy to study the asymptotic behavior near non-hyperbolic equilibria and periodic orbits; naturally, Sacker-Sell spectrum as well as the zero number control properties on the associated invariant spaces and local invariant manifolds {were} introduced;
moreover, we {proved} that the zero number function of the difference
of the solution under consideration and its limiting orbit does not change under space translation and time evolution
(see  Lemma \ref{auton-inhomo-cons}).

Additionally, Theorem \ref{aut-thm} shows that the structure of non-wandering set is esentially better than that of global attractor obtained in \cite{FRW,Ma-Na}.
\vskip 2mm

 \noindent {\bf Remark 3.} This remark is for the periodic system \eqref{general-equa}+\eqref{periodic-eq}.
 For the periodic system \eqref{general-equa}+\eqref{periodic-eq}, the results we obtained here pave the way to prove the Morse-Smale property of this system. Actually, if one assumes that all the points in the non-wandering set are hyperbolic, then obviously the non-wandering set contains only finitely many hyperbolic fixed points of the associated Poincar\'{e} map $P$. Moreover, by using similar arguments as in \cite{CCJ} or \cite{CR} (using the invariant nested cones method), it turns out to be not hard to obtain that the stable and unstable manifolds between connecting hyperbolic fixed points will intersect transversal. As a consequence, Morse-Smale property is expected to obtain.
\vskip 2mm

\noindent {\bf Remark 4.} This remark is about  scalar reaction diffusion equations with
separated boundary conditions. For the following reaction diffusion equation with separated boundary conditions,
\begin{equation}
\label{general-equa-1}
\begin{cases}
u_t=u_{xx}+f(t,u,u_x),\quad t>0,\,\, 0<x<2\pi\cr
(Bu)(t,0)=(Bu)(t,2\pi)=0,\quad t>0,
\end{cases}
\end{equation}
where $Bu=u$ or $Bu=u_x$,
long time behavior of bounded solutions have been thoroughly investigated
(see \cite{Hale} and the references therein). {For instance, it is proved  in \cite[Lemma 4.6]{CCJ} that, if \eqref{general-equa-1} is a time $T$-periodic system, then any non-wandering point of the associated Poincar\'{e} map lies in a $T$-periodic orbit of this system.}
 In contrast to the separated boundary condition cases, the  periodic boundary case we consider here is much more complicated. As a matter of fact, in the separated boundary condition cases, by using the non-increasing property of the zero number function for solutions of linear systems, the value $u(t,x)$  at $x=0$ or $2\pi$  naturally induces an order relation for sufficiently large time.
  Therefore, many already established results for \eqref{general-equa-1} are similar to those for one-dimensional ordinary differential  equations (see also \cite{Hale} and the literatures therein). However,
 no such order relation can be directly induced
 for the periodic boundary condition case in general, which makes periodic boundary condition case more difficult to study.

 Nevertheless, since the zero number function plays an essential role in dealing with these systems, we {also constructed} some order relation for the current system. In fact, we {constructed} a maximum order relation (see the proofs for Theorem \ref{aut-thm} and Theorem \ref{peri-thm}) between the orbit of a non-wandering point and the orbit of a non-wandering point push forwarded for some fixed time  based on three important  Lemmas \ref{fixed-constant}-\ref{auton-T-const} (resp. Lemmas \ref{period-center-constant}-\ref{peri-mT-const}) for system \eqref{general-equa}+\eqref{autonomous} (resp. \eqref{general-equa}+\eqref{periodic-eq}). It is this order relation that enables us to prove the main results stated in Theorem \ref{aut-thm} and Theorem \ref{peri-thm}. The maximum order relation we obtained here seems weaker than the order relation for separated boundary condition cases, since this order relation may only correct for that non-wandering points.  To construct this maximum order relation, we need to know the global property of solution of $u(t,x)$ on $S^1$ rather than it only on $0$ or $2\pi$. Therefore, the problem we  consider here requires more delicate analysis of the property of zero number function on the solution $u(t,x)$ with the help of the invariance of $S^1$-group action.
\vskip 2mm

\noindent {\bf Remark 5.} This last remark is about \eqref{general-equa} with general time dependent $f$.
 For general non-autonomous system, for instance, $f$ in \eqref{general-equa} is almost periodic in $t$, the authors of the current paper thoroughly investigated the structure of $\omega$-limit sets of \eqref{general-equa} {(in the sense of skew-product semiflow)}, in \cite{SWZ2,SWZ3,SWZ,SWZ4}. It was proved in \cite{SWZ4}, that any compact minimal invariant set can be residually embedded into an invariant set of some almost automorphically-forced flow on a circle. Moreover, the structure of the omega-limit set of any bounded orbit {was} given: it contains at most two minimal sets that cannot be obtained from each other by phase translation. A very interesting phenomenon was discovered, from a macro perspective, the structure of omega-limit set is simple; while from a microscopic perspective, the appearance of almost automotphically forced circle flow may make its dynamics very complicated. In some special cases (which also includes the equation on an interval with Dirichlet and Neumann boundary conditions), the flow on a minimal set topologically conjugates to an almost periodically-forced minimal flow on $\mathbb{R}$. The counterexample given in \cite{SWZ4} {showed} that even for quasi-periodic equations, these results {could} not be further improved in general. Therefore, the results in \cite{SWZ4} {revealed} that there are essential differences between time-periodic cases and non-periodic cases. One may also refer to \cite{SWZ2,SWZ3,SWZ,Zhou} for more delicate results about the structure of $\omega$-limit sets as well as minimal sets, in some special cases.

\section*{Appendix}

At the end of this article, we give an example to show that a non-wandering point of a general dynamical system  may not be limit point. This example is from MathOverflow constructed by Dr. Ilkka T\"{o}rm\"{a}.

Assume that $X=\{0,1\}^{\mathbb Z}$, then any point $x$ in $X$ can be written by $x=(x_i)_{i\in\mathbb{Z}}$, or by
\[
x=\cdots x_{-2}x_{-1}x_0x_1x_2\cdots,
\]
where each $x_i\in \{0,1\}$. The central $(2 k+1)-$ block of $x$ is $x_{[-k,k]}=x_{-k}x_{-k+1}\cdots x_k$. The metric $d$ on $X$ is defined as the following
$$
d(x, y)=\left\{\begin{array}{ll}{2^{-k}} & {\text { if } x \neq y \text { and } k \text { is maximal so that } x_{[-k, k]}=y_{[-k, k]}} \\ {0} & {\text { if } x=y}\end{array}\right.
$$
In other words, to measure the distance between $x$ and $y,$ we find the
largest $k$ for which the central $(2 k+1)$ -blocks of $x$ and $y$ agree, and use
$2^{-k}$ as the distance (with the conventions that if $x=y$ then $k=\infty$ and $2^{-\infty}=0$).

The map $\sigma$ on $X$ maps a point $x$ to
the point $y=\sigma(x)$ whose $i$-th coordinate is $y_{i}=x_{i+1}$. See the following picture for this map
\begin{equation*}
\begin{split}
  x \quad &=\cdots x_{-3}\,\,x_{-2}\,\,x_{-1}\,. x_{0}\,\,x_{1}\,\,x_{2}\,\,x_{3}\,\,\cdots \\
  \downarrow \sigma & \quad\quad\quad \, \,\, \swarrow\,\,\,\,\swarrow\,\,\,\,\swarrow\,\,\,\,\swarrow\,\,\,\swarrow\,\,\,\swarrow\\
 y=\sigma(x)&=\cdots  x_{-2}\,\,x_{-1}\,\,x_{0}\,\,. x_{1}\,\,x_{2}\,\,x_{3}\,\,\cdots
\end{split}
\end{equation*}
The composition of $\sigma$ with itself $k>0$ times $\sigma^k=\sigma\cdots\sigma$ shifts sequences $k$ places to the left, while $\sigma^{-k}=(\sigma^{-1})^k$ shifts the same amount to the right. It is easy to check that $(X,\sigma)$ is symbolic dynamical system (see \cite{LM} for the above conceptions in symbolic dynamical systems).

Consider a subspace $X_0\subset\{0,1\}^{\mathbb Z}$ defined by forbidden the words $10^{m}1^{n}0$ for all $m,n\in \mathbb{Z}^+$ (i.e., such blocks will not appear in $X_0$), then $(X_0,\sigma)$ is a subsystem of $(X,\sigma)$. By the definition of $X_0$, for any $x=(x_i)_{i\in\mathbb{Z}}\in X_0$, there exists $k\in \mathbb{Z}^+$ such that $\cdots =x_{-k-2}=x_{-k-1}=x_{-k}$ and $x_k=x_{k+1}=x_{k+2}=\cdots$. Thus, $X_0$ can only contain three classes of points:
\begin{itemize}
    \item [{\rm (i)}] $\cdots0001^n000\cdots$;
    \item [{\rm (ii)}] $\cdots 1110^n111\cdots$;
    \item [{\rm (iii)}]$\cdots 0001^n0^m111\cdots$;
\end{itemize}
where $m$, $n$ are non-negative integers.

Therefore, the only $\omega$-limit points of $(X_0,\sigma)$ are two uniform points: $\cdots 000.000\cdots$ and $\cdots111.111\cdots$. Choose $x^0=\cdots000.111\cdots$ and $x^n=\cdots000.1^n0^n111\cdots$, then $d(x^0,x^n)=2^{-n}$. Thus, $x^n$ can be arbitrary close to $x^0$ as $n\to \infty$. On the other hand, $\sigma^{2n}(x^n)=\cdots 0001^n0^n.111\cdots$ and $d(x^0,\sigma^{2n}(x^n))=2^{-n}\to 0$ as $n\to \infty$. Thus, $x^0$ is a non-wandering point, but is not a limit point.

\section*{Acknowledgments}
The authors would like to thank Siming Tu for informing us the above example and the reference \cite{LM}. {Dun Zhou would like to thank the Chinese Scholarship Council (201906845011) for its financial support during his overseas visit and express his gratitude to the Department of Mathematics and Statistics of Auburn University for  its hospitality.}


\begin{thebibliography}{10}
\bibitem{ABD}
F. Abdenur, C. Bonatti and  L. D\'{\i}az, {\rm Non-wandering sets with non-empty interiors}, Nonlinearity \textbf{17} (2004), no. 1, 175-191.

\bibitem{ALGM}
A. A. Andronov, E. A. Leontovich, I. I. Gordon and A. G. Ma\v{i}er, {\rm Theory of bifurcations of dynamic systems on a plane}, Translated from the Russian, 1973. xiv+482 pp.

\bibitem{2038390}
S. Angenent, \rm{The zero set of a solution of a parabolic equation}, J. Reine
 Angew. Math \textbf{390} (1988), 79--96.

\bibitem{Angenent1988}
S. Angenent and B. Fiedler, \rm{The dynamics of rotating waves in scalar reaction diffusion equations}, Trans. Amer. Math. Soc. \textbf{307} (1988), 545--568.

\bibitem{BCMN}
L. Block, E. M. Coven, I. Mulvey and Z. Nitecki, {\rm Homoclinic and non-wandering points for maps of the circle}, Ergodic Theory and Dynamical Systems, \textbf{3}(4) (1983), 521-532.

\bibitem{CCJ}
M. Chen, X. Chen and J. K. Hale, {\rm Structural stability for time-periodic one-dimensional parabolic equations},
J. Differential Equations \textbf{96} (1992), no. 2, 355-418.

\bibitem{Chen1989160}
X. Chen and H.~Matano, \rm{Convergence, asymptotic periodicity, and
  finite-point blow-up in one-dimensional semilinear heat equations}, J. Differential Equations \textbf{78} (1989), 160-190.

\bibitem{Chen-P}
X. Chen and P. Pol\'{a}\v{c}ik, \rm{Gradient-like structure and Morse decompositions for time-periodic one-dimensional parabolic equations}, J. Dynam. Differential Equations \textbf{7} (1995), no. 1, 73-107.

\bibitem{Chen98}
X. Chen, \rm{A strong unique continuation theorem for parabolic equations}, Math. Ann. \textbf{311} (1998), 603-630.

\bibitem{Chow-Leiva95}
S. Chow and H. Leiva, \rm{Existence and roughness of the exponential dichotomy for skew-product semiflow in Banach spaces}, J. Differential Equations \textbf{120} (1995), no. 2, 429-477.

\bibitem{Chow1994}
S. Chow and H. Leiva, \rm{Dynamical spectrum for time dependent linear systems in Banach spaces}, Japan J. Indust. Appl. Math., \textbf{11} (1994),379-415.

\bibitem{ChLuMa1} S. Chow, K. Lu, and J. Mallet-Paret,
{\rm Floquet theory for parabolic differential equations}, J. Differential Equations \textbf{109} (1994), no. 1, 147-200.

\bibitem{ChLuMa} S. Chow, K. Lu, and J. Mallet-Paret,
Floquet bundles for scalar parabolic equations, { Arch. Ration. Mech. Anal.}, {\bf 129} (1995), 245-304.


\bibitem{CN}
E. M. Coven and Z.Nitecki, {\rm Non-wandering sets of the powers of maps of the interval}, Ergodic Theory and Dynamical Systems, \textbf{1} (1) (1981), 9-31.


\bibitem{CR}
R. Czaja and C. Rocha, \rm{Transversality in scalar reaction-diffusion equations on a circle}, J. Differential Equations \textbf{245} (2008), 692-721.

\bibitem{Fiedler}
B.~Fiedler and J.~Mallet-Paret, \rm{A Poincar\'{e}-Bendixson theorem for
  scalar reaction diffusion equations}, Arch.
  Ration. Mech. Anal. \textbf{107} (1989), 325-345.

\bibitem{FRW}
B. Fiedler, C. Rocha and M. Wolfrum, {\rm Heteroclinic orbits between rotating waves of semilinear parabolic equations on the circle},
J. Differential Equations \textbf{201} (2004), no. 1, 99-138.

\bibitem{Hale88}
J.K. Hale, {\rm Asymptotic behavior of dissipative systems}, American Mathematical Soc., 2010.

\bibitem{Hale}
J.K. Hale, {\rm Dynamics of a scalar parabolic equation}, Canad. Appl. Math. Quart. \textbf{5} (1997), no. 3, 209-305.

\bibitem{Hen}
 D. Henry, \rm{Geometric Theory of Semilinear Parabolic Equations}, Lecture Notes Mathematics Vol.840,
New York, Springer, 1981.

\bibitem{Hess}
P. Hess, \rm{Periodic-parabolic boundary value problems and positivity}, Pitman Research Notes in Mathematics Series, 247.

\bibitem{HPPS}
M. Hirsch, J. Palis, C. Pugh and M.Shub, Neighborhoods of hyperbolic sets. Invent. Math. \textbf{9} (1970), 121-134.
\bibitem{JR1}
R. Joly and G. Raugel, \rm{Generic hyperbolicity of equilibria and periodic orbits of the
parabolic equation on the circle}, Trans. Amer. Math. Soc. \textbf{362}

\bibitem{JR2}
R. Joly and G. Raugel, \rm{Generic Morse-Smale property for the parabolic equation on
the circle}, Annal. Institute Henri Poincar\'{e}, Analyse Non Lin\'{e}aire \textbf{27} (2010),
1397-1440.
\bibitem{Katok-Has}
A. Katok and B. Hasselblatt, {\rm Introduction to the modern theory of dynamical systems}, 54. Cambridge University Press, Cambridge, 1995. xviii+802 pp.

\bibitem{LM}
D. Lind and B. Marcus, {\rm An introduction to symbolic dynamics and coding}, Cambridge University Press, 1995. xvi+495 pp.
\bibitem{Massatt1986}
P.~Massatt, {\rm The Convergence of solutions of scalar reaction diffusion
  equations with convection to periodic solutions}, Preprint, 1986.

\bibitem{H.MATANO:1982}
H.~Matano, \rm{Nonincrease of the lap-number of a solution for a one-dimensional semi-linear parabolic equation}, J. Fac. Sci. Univ. Tokyo Sect.IA. \textbf{29} (1982), 401-441.

\bibitem{Matano}
H.~Matano, \rm{Asymptotic behavior of solutions of semilinear heat equations on
  $S^1$}, Nonlinear Diffusion Equations and Their Equilibrium States II (W.-M.
  Ni, L.A. Peletier, and James Serrin, eds.), Mathematical Sciences Research
  Institute Publications, vol.~13, Springer US, 1988, 139-162.

\bibitem{Ma-Na}
H. Matano, K.-I. Nakamura, {\rm The global attractor of semilinear parabolic equations on $S^1$},
Discrete Contin. Dynam. Syst. \textbf{3} (1997) 1-24.

\bibitem{Mierczynski}
J. Mierczy\'nski and W. Shen, \rm{Spectral theory for random and nonautonomous parabolic equations and applications}, Chapman \& Hall/CRC Monographs and Surveys in Pure and Applied Mathematics, 139. CRC Press, 2008.

\bibitem{Palis1968}
J. Palis, {\rm On Morse-Smale dynamical systems}, Topology \textbf{8} (1968), 385-404.

\bibitem{Palis-Smale}
J. Palis and S. Smale, {\rm Structural stability theorems}. 1970 Global Analysis (Proc. Sympos. Pure Math., Vol. XIV, Berkeley, Calif., 1968) pp. 223-231 Amer. Math. Soc., Providence, R.I.

\bibitem{Pol}
P. Pol\'{a}\v{c}ik, Parabolic equations: Asymptotic behavior and dynamics on invariant manifolds, in: Handbook of Dynamical Systems, vol. 2, North-Holland, Amsterdam, 2002, pp. 835¨C883.

\bibitem{PS1970}
C. Pugh and M.Shub, {\rm The $\Omega$-stability theorem for flows}, Invent. Math. \textbf{11} (1970), 150-158.
\bibitem{Ruelle89}
D. Ruelle, \rm{Elements of differentiable dynamics and bifurcation theory}. Academic Press, Inc., Boston, MA, 1989. viii+187 pp.
\bibitem{Sacker1978}
R. Sacker and G. Sell, \rm{A spectral theory for linear differential systems}, J. Differential Equations \textbf{27} (1978), 320-358.

\bibitem{Sacker1991}
R. Sacker and G. Sell, \rm{Dichotomies for linear evolutionary equations in Banach spaces}, J. Differential Equations \textbf{113} (1994), 17-67.

\bibitem{SF1}
B.~Sandstede and B.~Fiedler, \rm{Dynamics of periodically forced parabolic
  equations on the circle}, Ergodic Theory and Dynamical Systems \textbf{12}
  (1992), 559--571.

\bibitem{SWZ2}
W. Shen, Y. Wang and D. Zhou, \rm{Structure of $\omega$-limit sets for almost-periodic parabolic equations on $S^1$ with reflection symmetry}, J. Differential Equations \textbf{267} (2016), 6633-6667.

\bibitem{SWZ3}
W. Shen, Y. Wang and D. Zhou, \rm{Long-time behavior of almost periodically forced parabolic equations on the circle}
J. Differential Equations \textbf{266} (2019), 1377-1413.

\bibitem{SWZ}
W. Shen, Y. Wang and D. Zhou, \rm{Almost automorphically and almost periodically forced circle flows of almost periodic parabolic equations on $S^1$}, J. Dynam. Differential Equations (2019), https://doi.org/10.1007/s10884-019-09786-7.

\bibitem{SWZ4}
W. Shen, Y. Wang and D. Zhou, \rm{Almost automorphically forced flows on $S^1$ or $\mathbb{R}$ in one-dimensional almost periodic semilinear heat equations}, https://arxiv.org/abs/2007.05532, Preprint 28p.

\bibitem{Smale}
S. Smale, {\rm Differentiable dynamical systems}, Bull. Amer. Math. Soc. \textbf{73} (1967), 747-817.

\bibitem{Smale1968}
S. Smale,{\rm The $\Omega$-stability theorem}, 1970 Global Analysis (Proc. Sympos. Pure Math., Vol. XIV, Berkeley, Calif., 1968) pp. 289-297 Amer. Math. Soc., Providence, R.I.

\bibitem{wiggins}
S. Wiggins, {\rm Introduction to applied nonlinear dynamical systems and chaos}. Springer, 2003.
\bibitem{Will}
R. F. Williams, {\rm One-dimensional non-wandering sets}, Topology \textbf{6} (1967), 473-487.

\bibitem{Ye}
X. Ye, {\rm Non-wandering points and the depth of a graph map}, J. Austral. Math. Soc. Ser. A \textbf{69} (2000), 143-152.
\bibitem{Zhou}
D. Zhou, {\rm Lifting properties of minimal sets for parabolic equations on $S^1$ with reflection symmetry}, Proc. Amer. Math. Soc. \textbf{145} (2017) 1175-1185
\end{thebibliography}
\end{document}